\documentclass[12pt]{amsart}
\usepackage{}
\usepackage{amsmath}
\usepackage{amsfonts}
\usepackage{amssymb}
\usepackage[all]{xy}           
\usepackage{extarrows}

\usepackage{bbding}
\usepackage{txfonts}
\usepackage[shortlabels]{enumitem}
\usepackage{ifpdf}
\ifpdf
\usepackage[colorlinks,final,backref=page,hyperindex]{hyperref}
\else
\usepackage[colorlinks,final,backref=page,hyperindex,hypertex]{hyperref}
\fi
\usepackage{tikz}
\usepackage[active]{srcltx}

\newcommand{\deleted}[1]{}
\newcommand{\delete}[1]{}
\newcommand{\mynote}[1]{}
\newcommand\notes[1]{}

\newcommand\changed[1]{#1}
\changed{}

\newtheorem{theorem}{Theorem}[section]

\newtheorem{lemma}[theorem]{Lemma}
\newtheorem{coro}[theorem]{Corollary}

\newtheorem{prop}[theorem]{Proposition}
\theoremstyle{definition}
\newtheorem{defn}[theorem]{Definition}
\newtheorem{remark}[theorem]{Remark}
\newtheorem{exam}[theorem]{Example}
\deleted{\newtheorem{theorem}{Theorem}[section]
\newtheorem{prop-def}{Proposition-Definition}[section]
\newtheorem{coro-def}{Corollary-Definition}[section]

}

\newcommand{\nc}{\newcommand}
\nc{\tred}[1]{\textcolor{red}{#1}}
\nc{\tblue}[1]{\textcolor{blue}{#1}}
\nc{\tgreen}[1]{\textcolor{green}{#1}}
\nc{\tpurple}[1]{\textcolor{purple}{#1}}
\nc{\btred}[1]{\textcolor{red}{\bf #1}}
\nc{\btblue}[1]{\textcolor{blue}{\bf #1}}
\nc{\btgreen}[1]{\textcolor{green}{\bf #1}}
\nc{\btpurple}[1]{\textcolor{purple}{\bf #1}}

\renewcommand{\Bbb}{\mathbb}

\newcommand{\efootnote}[1]{}

\renewcommand{\textbf}[1]{}

\nc{\mlabel}[1]{\label{#1}}  
\nc{\mcite}[1]{\cite{#1}}  
\nc{\mref}[1]{\ref{#1}}  
\nc{\mbibitem}[1]{\bibitem{#1}} 

\delete{
\nc{\mlabel}[1]{\label{#1}  
{\hfill \hspace{1cm}{\bf{{\ }\hfill(#1)}}}}
\nc{\mcite}[1]{\cite{#1}{{\bf{{\ }(#1)}}}}  
\nc{\mref}[1]{\ref{#1}{{\bf{{\ }(#1)}}}}  
\nc{\mbibitem}[1]{\bibitem[\bf #1]{#1}} 
}

\renewcommand\geq{\geqslant}
\renewcommand\leq{\leqslant}

\renewcommand\bar[1]{\overline{#1}}
\renewcommand\tilde[1]{\widetilde{#1}}

\nc{\lead}{\mathrm{Lead}}
\nc{\Id}{\mathrm{Id}}
\nc{\Irr}{\mathrm{Irr}}
\nc{\End}{\mathrm{End}}
\nc{\vx}{\sigma}
\nc{\vy}{\tau}
\nc{\dvx}{\sigma^{(1)}}
\nc{\dvy}{\tau^{(1)}}
\nc{\done}{\vep}
\nc{\citep}[1]{\cite{#1}}
\nc{\wt}{\mathrm{wt}}
\nc{\bre}[1]{|#1|} \nc{\mapmonoid}{\frakM}
\nc{\disjoint}{\frakM'}
\nc{\ncpoly}[1]{\langle #1\rangle}  
\nc{\mapm}[1]{\lfloor\!|{#1}|\!\rfloor}             
\nc{\diff}[1]{{}^\NC\{ #1 \}}
\nc{\disj}[1]{\{{#1}\}'}
\nc{\mdisj}[1]{\frakM'(#1)}
\nc{\brho}{\bar{\rho}}
\nc{\om}{\bar{\frakm}}
\nc{\frakn}{\mathfrak n}
\nc{\ddeg}[1]{^{(#1)}}
\nc{\opset}{X} \nc{\genset}{{Z}}
\nc{\NC}{\mathrm{{NC}}}
\nc{\leaf}{\mathrm{leaf}}
\nc{\twig}{\mathrm{twig}}
\nc{\fe}{\mathrm{fl}}
\nc{\munderline}[1]{#1}
\nc{\bo}{o}
\nc{\ofe}{\mathrm{ofl}}
\nc{\dfe}{\mathrm{dfe}}
\nc{\fex}{\mathrm{fex}}
\nc{\dl}{\mathrm{dlex}}
\nc{\db}{\mathrm{db}}

\nc{\bin}[2]{ (_{\stackrel{\scs{#1}}{\scs{#2}}})}  
\nc{\bs}{\bar{S}}
\nc{\cosum}{\sqsubset}
\nc{\la}{\longrightarrow}
\nc{\rar}{\rightarrow}
\nc{\dar}{\downarrow}
\nc{\dprod}{**}
\nc{\dap}[1]{\downarrow \rlap{$\scriptstyle{#1}$}}
\nc{\md}{\mathrm{dth}}
\nc{\uap}[1]{\uparrow \rlap{$\scriptstyle{#1}$}}
\nc{\defeq}{\stackrel{\rm def}{=}}
\nc{\disp}[1]{\displaystyle{#1}}
\nc{\dotcup}{\ \displaystyle{\bigcup^\bullet}\ }
\nc{\gzeta}{\bar{\zeta}}
\nc{\hcm}{\ \hat{,}\ }
\nc{\hts}{\hat{\otimes}}

\nc{\free}[1]{\bar{#1}}
\nc{\uni}[1]{\tilde{#1}}
\nc{\hcirc}{\hat{\circ}}
\nc{\leng}{\ell}
\nc{\lleft}{[}
\nc{\lright}{]}
\nc{\lc}{\lfloor}
\nc{\rc}{\rfloor}
\nc{\lb}{[} 
\nc{\rb}{]} 
\nc{\curlyl}{\left \{ \begin{array}{c} {} \\ {} \end{array}
    \right .  \!\!\!\!\!\!\!}
\nc{\curlyr}{ \!\!\!\!\!\!\!
    \left . \begin{array}{c} {} \\ {} \end{array}
    \right \} }
\nc{\longmid}{\left | \begin{array}{c} {} \\ {} \end{array}
    \right . \!\!\!\!\!\!\!}
\nc{\onetree}{\bullet}
\nc{\ora}[1]{\stackrel{#1}{\rar}}
\nc{\ola}[1]{\stackrel{#1}{\la}}

\nc{\mot}{{{\boxtimes\,}}}
\nc{\otm}{\overline{\boxtimes}}
\nc{\sprod}{\bullet}
\nc{\scs}[1]{\scriptstyle{#1}}
\nc{\mrm}[1]{{\rm #1}}
\nc{\msum}{\sum\limits}
\nc{\margin}[1]{\marginpar{\rm #1}}   
\nc{\dirlim}{\displaystyle{\lim_{\longrightarrow}}\,}
\nc{\invlim}{\displaystyle{\lim_{\longleftarrow}}\,}
\nc{\mvp}{\vspace{0.3cm}}
\nc{\tk}{^{(k)}}
\nc{\tp}{^\prime}
\nc{\ttp}{^{\prime\prime}}
\nc{\svp}{\vspace{2cm}}
\nc{\vp}{\vspace{8cm}}
\nc{\proofbegin}{\noindent{\bf Proof: }}
\nc{\proofend}{$\blacksquare$ \vspace{0.3cm}}
\nc{\modg}[1]{\!<\!\!{#1}\!\!>}
\nc{\intg}[1]{F_C(#1)}
\nc{\lmodg}{\!<\!\!}
\nc{\rmodg}{\!\!>\!}
\nc{\cpi}{\widehat{\Pi}}
\nc{\sha}{{\mbox{\cyr X}}}  
\nc{\shap}{{\mbox{\cyrs X}}} 
\nc{\shpr}{\diamond}    
\nc{\shp}{\ast}
\nc{\shplus}{\shpr^+}
\nc{\shprc}{\shpr_c}    
\nc{\msh}{\ast}
\nc{\zprod}{m_0}
\nc{\oprod}{m_1}
\nc{\vep}{\varepsilon}
\nc{\labs}{\mid\!}
\nc{\rabs}{\!\mid}

\nc{\dth}{d}
\nc{\mmbox}[1]{\mbox{\ #1\ }}
\nc{\fp}{\mrm{FP}} \nc{\rchar}{\mrm{char}} \nc{\Fil}{\mrm{Fil}}
\nc{\Mor}{Mor\xspace}
\nc{\gmzvs}{gMZV\xspace}
\nc{\gmzv}{gMZV\xspace}
\nc{\mzv}{MZV\xspace}
\nc{\mzvs}{MZVs\xspace}
\nc{\id}{\mrm{id}} 
\nc{\incl}{\mrm{incl}} \nc{\map}{\mrm{Map}} \nc{\mchar}{\rm char}
\nc{\nz}{\rm NZ} \nc{\supp}{\mathrm Supp}

\nc{\Alg}{\mathbf{Alg}}
\nc{\Bax}{\mathbf{Bax}}
\nc{\bff}{\mathbf f}
\nc{\bfk}{{\bf k}}
\nc{\bfone}{{\bf 1}}
\nc{\bfx}{\mathbf x}
\nc{\bfy}{\mathbf y}
\nc{\base}[1]{\bfone^{\otimes ({#1}+1)}} 
\nc{\Cat}{\mathbf{Cat}}
\delete{}
\nc{\detail}{\marginpar{\bf More detail}
    \noindent{\bf Need more detail!}
    \svp}
\nc{\Int}{\mathbf{Int}}
\nc{\Mon}{\mathbf{Mon}}
\nc{\rbtm}{{shuffle }}
\nc{\rbto}{{Rota-Baxter }}
\nc{\remarks}{\noindent{\bf Remarks: }}
\nc{\Rings}{\mathbf{Rings}}
\nc{\Sets}{\mathbf{Sets}}

\nc{\BA}{{\Bbb A}} \nc{\CC}{{\Bbb C}} \nc{\DD}{{\Bbb D}}
\nc{\EE}{{\Bbb E}} \nc{\FF}{{\Bbb F}} \nc{\GG}{{\Bbb G}}
\nc{\HH}{{\Bbb H}} \nc{\LL}{{\Bbb L}} \nc{\NN}{{\Bbb N}}
\nc{\KK}{{\Bbb K}} \nc{\QQ}{{\Bbb Q}} \nc{\RR}{{\Bbb R}}
\nc{\TT}{{\Bbb T}} \nc{\VV}{{\Bbb V}} \nc{\ZZ}{{\Bbb Z}}

\nc{\cala}{{\mathcal A}} \nc{\calc}{{\mathcal C}}
\nc{\cald}{{\mathcal D}} \nc{\cale}{{\mathcal E}}
\nc{\calf}{{\mathcal F}} \nc{\calg}{{\mathcal G}}
\nc{\calh}{{\mathcal H}} \nc{\cali}{{\mathcal I}}
\nc{\call}{{\mathcal L}} \nc{\calm}{{\mathcal M}}
\nc{\caln}{{\mathcal N}} \nc{\calo}{{\mathcal O}}
\nc{\calp}{{\mathcal P}} \nc{\calr}{{\mathcal R}}
\nc{\cals}{{\mathcal S}}
\nc{\calt}{{\mathcal T}} \nc{\calw}{{\mathcal W}}
\nc{\calk}{{\mathcal K}} \nc{\calx}{{\mathcal X}}
\nc{\CA}{\mathcal{A}}

\nc{\fraka}{{\mathfrak a}} \nc{\frakA}{{\mathfrak A}}
\nc{\frakb}{{\mathfrak b}} \nc{\frakB}{{\mathfrak B}}
\nc{\frakD}{{\mathfrak D}} \nc{\frakH}{{\mathfrak H}}
\nc{\frakM}{{\mathfrak M}} \nc{\bfrakM}{\overline{\frakM}}
\nc{\frakm}{{\mathfrak m}} \nc{\frakP}{{\mathfrak P}}
\nc{\frakN}{{\mathfrak N}} \nc{\frakp}{{\mathfrak p}}
\nc{\frakS}{{\mathfrak S}} \nc{\frakx}{{\mathfrak x}}
\nc{\ox}{\bar{\frakx}} \nc{\frakX}{{\mathfrak X}}
\nc{\fraky}{{\mathfrak y}} \nc\dop{\delta}

\font\cyr=wncyr10
\font\cyrs=wncyr7
\nc{\redt}[1]{\textcolor{red}{#1}}

\nc{\Xiaosong}[1]{\textcolor{red}{\tt \underline{Xiaosong:}#1}}
\nc{\Shilong}[1]{\textcolor{blue}{\tt \underline{Shilong:}#1}}

\nc{\N}{N}
\nc{\NM}{N_M}
\nc{\A}{A}
\nc{\Hom}{{\rm Hom}}
\nc{\Ker}{{\rm Ker}}
\nc{\IIm}{{\rm Im}}
\nc{\NR}{U_{\N}(\A)}
\nc{\ot}{\otimes}

\nc{\rbo}{R_{RB}\langle Q \rangle}
\nc{\rbeta}{(R,P)} 
\nc{\salpha}{(S,\alpha)} 
\nc{\tgamma}{(T,\gamma)}
\nc{\otr}{\ot_{\rbeta}} 
\nc{\Free}{F_1} 
\nc{\rf}{\tilde{F}}
\nc{\revise}[1]{\textcolor{red}{#1}}

\begin{document}
\title[Nijenhuis modules and the ring of Nijenhuis operators]{Nijenhuis modules and the ring of Nijenhuis operators}

\author{Shilong Zhang}
\address{College of Science, Northwest A$\&$F University, Yangling 712100, Shaanxi, P.\,R. China}
\email{shlz@nwafu.edu.cn}

\author{Xiao-Song Peng$^{\ast}$}
\address{School of Mathematics and Statistics,
Jiangsu Normal University, Xuzhou, Jiangsu 221116, P.\,R. China}
\email{pengxiaosong3@163.com}
\thanks{$^{\ast}$ the corresponding author}

\author{Liuqing Dong}
\address{College of Science, Northwest A$\&$F University, Yangling 712100, Shaanxi, P.\,R. China}
\email{2024014490@nwafu.edu.cn}


\date{\today}
\begin{abstract}
In this paper, we study the Nijenhuis modules of Nijenhuis algebras. The concepts of free, projective, injective, and flat Nijenhuis modules are introduced, and a construction of free Nijenhuis modules is given. The ring of Nijenhuis operators is introduced, and the relationship between Nijenhuis modules and modules over this ring is established by proving an isomorphism of the corresponding categories. Finally, it is proved that the category of Nijenhuis modules has enough projective, injective, and flat objects.
\end{abstract}

\subjclass[2020]{16D40, 16S10, 16W99}

\keywords{
Nijenhuis algebra; Nijenhuis module; projective module;  injective module; flat module; ring of Nijenhuis operators}

\date{\today}

\maketitle

\tableofcontents

\setcounter{section}{0}


\section{Introduction}
The notion of Nijenhuis operators originated from differential geometry, where Nijenhuis introduced a tensor field whose vanishing Nijenhuis torsion characterizes complex structures on manifolds~\mcite{Nij51}. Later, this concept was adapted to Lie algebras~\mcite{KM} and associative algebras~\mcite{CGM}, leading to the study of a Nijenhuis algebra as an algebra $A$ equipped with a linear operator $\N: A \to A$ satisfying the Nijenhuis identity
\begin{align*}
\N(a) \N(b)=\N (\N(a)b+a\N(b)-\N(ab)), \, \text{ for any } a, b \in A.
\end{align*}
 In recent years, Nijenhuis operators have attracted increasing attention due to their connections with NS-algebras~\mcite{LG, Ler04}, N-dendriform algebras~\mcite{LG}, twisted Rota–Baxter operators~\mcite{Uch08}, and deformation theory~\mcite{Das24}.

The recent researches of Nijenhuis algebras have largely followed those of Rota-Baxter algebras. Commutative Nijenhuis algebras were
constructed in~\mcite{Ebr04, EL} following the construction of free commutative Rota-Baxter
algebras~\mcite{GK00}. Similar to the Rota-Baxter universal enveloping algebra of a (tri-)dendriform algebra constructed in~\mcite{EG08}, the universal enveloping Nijenhuis algebra of an NS algebra was given in~\mcite{LG}. The representation theory of Rota-Baxter algebras were studied in~\mcite{GL,QGG} and recently it was generalized to the multiple case~\mcite{HPZ}. Despite the growing interest in Nijenhuis algebras, their representation theory remains relatively underdeveloped compared to the well-studied Rota–Baxter algebras. In particular, fundamental categorical notions such as projective, injective, and flat modules in the context of Nijenhuis algebras have not been systematically explored, which hinders the application of homological methods to the study of Nijenhuis structures.

In this paper, we study the category of Nijenhuis modules, including the free, projective, injective and flat Nijenhuis modules. We show that there are enough projective, injective and flat objects in this category, enabling us to define the derived functors. More precisely, we construct free Nijenhuis modules explicitly via free operated modules, and characterize when a Nijenhuis algebra itself becomes a free Nijenhuis module. Furthermore, we introduce the ring of Nijenhuis operators $U_N(A)$, which plays a role analogous to the universal enveloping algebra for Lie algebras. We prove that the category of left Nijenhuis modules is isomorphic to the category of left $U_N(A)$-modules, thereby reducing the study of Nijenhuis modules to classical module theory. As an application, we establish the existence of enough projective and injective Nijenhuis modules via the corresponding properties of $U_N(A)$-modules. Finally, we define the tensor product of Nijenhuis modules and investigate flatness, showing that there are enough flat Nijenhuis modules as well. These results lay the homological foundation for further investigations, such as the construction of derived functors in the category of Nijenhuis modules.

{\bf The paper is organized as follows.} In Section~\mref{sec:free}, we introduce the concept of Nijenhuis modules and give basic examples. We then construct free Nijenhuis modules via free operated modules and establish a restricted condition under which a Nijenhui algebra is a free Nijenhui module.  Section~\mref{sec:ringCon} is devoted to the ring of Nijenhuis operators $U_N(A)$. We prove that the category of left Nijenhuis modules is isomorphic to the category of left $U_N(A)$-modules and give a general construction of the ring of Nijenhuis operators. In Section~\mref{sec:pi}, we study projective and injective Nijenhuis modules. We show that there are enough projective and injective Nijenhuis modules. Section~\mref{sec:flat} introduces the tensor product of Nijenhuis modules and investigates flatness. We prove that there are enough flat Nijenhuis modules.

{\bf Notation.}
Throughout this paper, let $\bfk$ be a unitary commutative ring unless the contrary is specified, which will be the base ring of all modules, algebras, coalgebras, bialgebras, tensor products, as well as linear maps.

\section{Nijenhuis module and free Nijenhuis modules}
\mlabel{sec:free}

In this section, we introduce the concept of Nijenhuis modules and give some basic properties of Nijenhuis modules. Then we give a construction of free Nijenhuis modules through free operated modules. Finally, we give a restricted condition under which a Nijenhuis algebra is a free Nijenhuis module over itself.

\subsection{Nijenhuis modules} We first recall the notion of Nijenhuis algebra before introducing the concept of Nijenhuis modules.

\begin{defn}
A {\bf Nijenhuis algebra} $(A,\N)$ is an associative algebra $A$ together with a linear map $\N: A \to A$ such that 
\begin{align*}
\N(a)\N(b)=\N\Big(\N(a)b+a\N(b)-\N(ab)\Big), \ \text{for any} \, a,b \in A.
\end{align*}
The linear map $\N$ is called {\bf a Nijenhuis operator} on $A$.
\end{defn}

\begin{exam}\mcite{Das24}
Let $A$ be an associative algebra. Then
\begin{enumerate}
\item The identity map $\id_A: A \to A$ is a Nijenhuis operator on $A$; 
\item For any element $x \in A$, the left multiplication map $l_x: A \to A, a \mapsto xa$ and the right multiplication map $r_x:A \to A, a \mapsto ax$ are both Nijenhuis operators on $A$;
\item If $\N: A \to A$ is a Nijenhuis operator on $A$ and $\lambda \in {\bf k}$, then $\lambda \N$ is also a Nijenhuis operator on $A$.
\end{enumerate}
\end{exam}

\begin{defn}
Let $(A,\N)$ be a Nijenhuis algebra. A {\bf left (Nijenhuis) $(A,\N)$-module} $(M,\NM)$ is a
left $A$-module $M$ together with a linear operator $\NM: M \to M$ such that
\begin{align}
\N(a)\NM(m)=\NM(\N(a)m + a\NM(m)-\NM (am)), \ \text{for any } a\in A, m\in M. \mlabel{eq:lnm}
\end{align}
\end{defn}

\begin{defn}
Let $(A,\N)$ be a Nijenhuis algebra.  For two left $(A,\N)$-modules $(M,\NM)$ and $(M',\N_{M'})$, a {\bf left $(A,\N)$-module homomorphism} from $(M, \NM)$ to $(M',\NM')$ is a left $A$-module homomorphism $f: M \to M'$ such that $f \circ \NM=\N_{M'} \circ f$. If moreover $f$ is a bijection, then $(M,\N_M)$ is {\bf isomorphic to} $(M',\N_{M'})$, denoted by $(M,\N_M) \cong (M',\N_{M'})$. 
\end{defn}

Denote by $ _{(A,\N)}{\bf Mod}$ the category of left $(A,\N)$-modules, with its objects the left $(A,\N)$-modules and its morphisms the left $(A,\N)$-module homomorphisms. Now we give some examples of left Nijenhuis modules.

\begin{exam}\mlabel{exam:njm}

\begin{enumerate}
\item Let $(A,\N)$ be a Nijenhuis algebra. Then, similar to the way an associative algebra is a left module over itself,  $(A,\N)$ is a left $(A,\N)$-module.

\item For any Nijenhuis algebra $(A, \N)$ and a left $A$-module $M$, $(M,\id_M)$ is a left $(A,\N)$-module.

\item For any Nijenhuis algebra $(A,\N)$ and a left $A$-module $M$, $(M,0_M)$ is a left $(A,\N)$-module where $0_M: M \to M, m \mapsto 0$.

\item For the Nijenhuis algebra $(A, l_x)$ with $x \in A$ and a left $A$-module $M$, define the linear map $l'_x: M \to M, m \mapsto x m$, then $(M,l'_x)$ is a left $(A, l_x)$-module.

\item For two left $(A,\N)$-modules $(M,\N_M)$ and $(M',\N_{M'})$, $(M \oplus M', \N_{M \oplus M'})$ is also a left $(A,\N)$-module, where $\N_{M \oplus M'}: M \oplus M' \to M \oplus M', (m,m') \mapsto (\N_{M}(m), \N_{M'}(m'))$.
\end{enumerate}
\end{exam}

\begin{defn}
Let $(A,\N)$ be a Nijenhuis algebra and $(M,\NM)$ a left $(A,\N)$-module. A {\bf left $(A,\N)$-submodule} of $(M,\NM )$ is a pair $(M_1, \N_{M_1})$, where $M_1$ is a submodule of the $A$-module $M$ and $\N_{M_1}$ is the restriction of $\NM$ to $M_1$ satisfying that $\N_{M_1}(M_1) \subseteq M_1$. 
\end{defn}

Now we define a quotient Nijenhuis module of a left Nijenhuis module as follows.
\begin{lemma}
Let $(A,\N )$ be a Nijenhuis algebra, $(M,\NM)$ a left $(A,\N)$-module and $(M',\N_{M'})$ a submodule of $(M,\NM)$. The pair
$(M/M',\N_{M/M'})$ is a left $(A,\N)$-module, where
$$\N_{M/M'}: M/M' \to M/M', \ m+M' \mapsto \NM(m)+M'.$$
We call $(M/M', \N_{M/M'})$ the {\bf quotient $(A,\N)$-module} of $(M,\NM )$ by $(M', \N_{M'})$.
\mlabel{lem:quotmod}
\end{lemma}

\begin{proof}
For $m_1+M', m_2+M' \in M/M'$, if $m_1+M'=m_2+M'$, then
$m_1-m_2 \in M'$ and note that  $\NM(M')=\N_{M'}(M') \subseteq M'$, hence $\NM(m_1)-\NM(m_2)=\NM(m_1-m_2) \in M'$, i.e., $\NM(m_1)+M'=\NM(m_2)+M'$. Thus $\N_{M/M'}$ is a well-defined map. Next we show that $\N_{M/M'}$ satisfies Eq.~(\mref{eq:lnm}).
For any $a\in A$, $m\in M$, we have
\begin{align*}
&\  \N(a) \N_{M/M'}(m+M') =\N(a)(\NM(m)+M') = \N(a)\NM(m) + M' \\
=&\ \NM(\N(a)m)+\NM(a \NM(m))-\NM(\NM(am)) + M' \\
=&\ (\NM(\N(a)m)+M') + (\NM(a \NM(m))+M') - (\NM (\NM(am)) + M') \\
=&\ \N_{M/M'}(\N(a)(m+M'))+\N_{M/M'}(a\N_{M/M'}(m+M'))- \N_{M/M'}(\N_{M/M'}(a(m+M'))),
\end{align*}
as required.
\end{proof}

\begin{remark}
Let $(\A,\N)$ be a Nijenhuis algebra. A left ideal $L$ of $A$ is called a {\bf left Nijenhuis ideal} if $\N(L) \subseteq L$. Then a left Nijenhuis ideal $(L,\N|_L)$ is a left $(A,\N)$-module. By Lemma~\mref{lem:quotmod}, $(A/L, \N_{A/L})$ is also a left $(A,\N)$-module, where
\[
\N_{A/L}: A/L \rightarrow A/L, \, a+L \mapsto \N(a)+L.
\]
\end{remark}

\begin{prop}\mlabel{prop:kernel}
Let $(A,\N)$ be a Nijenhuis algebra. Let $f: (M,\N_M) \to (M',\N_{M'})$ be a left $(A,\N)$-module homomorphism. Then 
\begin{enumerate}
\item \mlabel{it:1} The pair $(\Ker f, \N_{M}|_{\Ker f})$ is a left $(A,\N)$-submodule of $(M,\N_M)$ and $(M/\Ker f, \N_{M/\Ker f})$ is a left $(A,\N)$-module where 
\[
\N_{M/\Ker f}: M/\Ker f \to M/\Ker f, \ m+\Ker f \mapsto \N_M(m)+\Ker f.
\]
\item \mlabel{it:2}  The pair  $(\IIm f, \N_{M'}|_{\IIm f})$ is a left $(A,\N)$-submodule of $(M',\N_{M'})$ and $(M'/\IIm f, \N_{M'/\IIm f})$ is a left $(A,\N)$-module where
\[
\N_{M'/\IIm f}: M'/\IIm f \to M'/\IIm f, \ m'+\IIm f \mapsto \N_{M'}(m')+\IIm f.
\]

\item \mlabel{it:3} The left $(A,\N)$-modules $(M/\Ker f, \N_{M/\Ker f})$ and $(\IIm f, \N_{M'}|_{\IIm f})$ are isomorphic.
\end{enumerate}
\end{prop}

\begin{proof} 

(\mref{it:1}) Recall that  $\Ker f$ is a left $A$-submodule of $M$. To verify that $(\Ker f, \N_{M}|_{\Ker f})$ is a left $(A,\N)$-submodule of $(M,\N_M)$, we only need to show that $\N_M(\Ker f) \subseteq \Ker f$. For any $m \in \Ker f$,
\[
f(\N_M(m))=\N_{M'}(f(m))=\N_{M'}(0)=0,
\]
thus $\N_M(m) \in \Ker f$ and $\N_M(\Ker f) \subseteq \Ker f$ as desired. By Lemma~\mref{lem:quotmod}, $(M/\Ker f, \N_{M/\Ker f})$ is a left $(A,\N)$-module.

(\mref{it:2})  Similarly, $\IIm f$ is a left $A$-submodule of $M'$. To prove that $(\IIm f, \N_{M'}|_{\IIm f})$ is a left $(A,\N)$-submodule of $(M',\N_{M'})$, We only need to confirm that $\N_{M'}(\IIm f) \subseteq \IIm f$. For any $m' \in \IIm f$, assume $f(m)=m'$ with $m \in M$. It follows that:
\[
\N_{M'}(m')=\N_{M'}(f(m))=f(\N_M(m)).
\]
Thus $\N_{M'}(m') \in \IIm f$, which means $\N_{M'}(\IIm f) \subseteq \IIm f$. Again by Lemma~\mref{lem:quotmod}, $(M'/\IIm f, \N_{M'/\IIm f})$ is indeed  a left $(A,\N)$-module. 

(\mref{it:3}) The canonical map $\overline{f}:M/\Ker f \to \IIm f, m+\Ker f \mapsto f(m)$ is an isomorphism of left A-modules. Furthermore, for any coset $m+\Ker f\in  M/\Ker f,$ we compute:
\[
\overline{f}(\N_{M/\Ker f}(m+\Ker f))=\overline{f}(\N_M(m)+\Ker f)=f(\N_M(m))=\N_{M'}(f(m))=\N_{M'} (\overline{f}(m+\Ker f)).
\]
This shows that $\overline{f} \circ \N_{M/\Ker f}= \N_{M'}|_{\IIm f} \circ \overline{f}$, so $\overline{f}$ is an isomorphism of left $(A,\N)$-modules.
\end{proof}

Given two left $(A,\N)$-modules $(M,\NM)$ and $(M',\N_{M'})$, let $\Hom_{(A,\N)}((M,\NM), (M',\N_{M'}))$ denote the set of $(A,\N)$-homomorphisms from $(M,\NM)$ to $(M',\N_{M'})$. Then we have

\begin{prop}\mlabel{prop:hom}
Let $(A,\N)$ be a Nijenhuis algebra and let $(M,\NM), (M',\N_{M'})$ be two left $(A,\N)$-modules. Then $\Hom_{(A,\N)}((M,\NM),(M',\N_{M'}))$ forms an abelian subgroup of $\Hom_A(M,M')$.
\end{prop}

\begin{proof}
 First, it is clear that  $\Hom_{(A,\N)}((M,\NM),(M',\N_{M'})) \subseteq \Hom_A(M,M')$ and the zero map is a left $(A,\N)$-homomorphism from $(M,\NM)$ to $(M',\N_{M'})$ (hence belongs to this subset). Let $f,g \in \Hom_{(A,\N)}((M,\NM),(M',\N_{M'}))$  and take an arbitrary element  $m \in M$. We have
\begin{align*}
&\ ((f-g) \circ \NM)(m)=(f-g)(\NM(m))=f(\NM(m))-g(\NM(m))\\
=&\ \N_{M'}(f(m))-\N_{M'}(g(m))=\N_{M'}(f(m)-g(m))=(\N_{M'} \circ (f-g))(m),
\end{align*}
which implies $f - g \in \Hom_{(A,\N)}((M,\N_M), (M',\N_{M'}))$. Therefore, $\Hom_{(A,\N)}((M,\N_M), (M',\N_{M'}))$ is closed under subtraction and contains the identity element (zero map) of $\Hom_A(M,M')$. Since $\Hom_A(M,M')$ is an abelian group under pointwise addition, it follows that $\Hom_{(A,\N)}((M,\N_M), (M',\N_{M'}))$ is an abelian subgroup of $\Hom_A(M,M')$. 
\end{proof} 

\begin{remark}
Combining Example~\mref{exam:njm} with Propositions~\mref{prop:kernel} and \mref{prop:hom}, we conclude that the category $_{(A,\N)}{\bf Mod}$ is an abelian category. Note that every $A$-module $M$ is an $(A,\N)$-module $(M,0)$ (with the zero Nijenhuis operator), which defines a full embedding of the category $_A {\bf Mod}$ (of left $A$-modules) into $_{(A,\N)}{\bf Mod}$. Moreover, composing this embedding with the forgetful functor $_{(A,\N)}{\bf Mod} \to  _A {\bf Mod}$ yields the identity functor on $_A {\bf Mod}$.
\end{remark} 

Recall~\mcite{Das24} that for a Nijenhuis algebra $(\A,\N)$, the underlying space $A$ inherites a new associative algebra structure with the multiplication 
\[
a \cdot_N b:=N(a)b+aN(b)-N(ab), \, \text{for $a,b \in \A$.}
\]
Denote this new algebra by $A_N$. Moreover, we have 

\begin{prop}~\mcite{CGM}
Let $(\A,\N)$ be a Nijenhuis algebra. Then

\begin{enumerate}
\item For each $k \geq 0$, $(A,\N^k)$ is also a Nijenhuis algebra.
\item For any $k, \ell \geq 0$, $(A_{N^k}, \N^{\ell})$ is also a Nijenhuis algebra.
\end{enumerate}
\end{prop}
Now for a left Nijenhuis module $(M,\NM)$ of $(\A,\N)$, define a new action of $A$ on $M$ by 
\[
a \ast_N m :=\N(a)m+a \NM(m)-\NM(am), \, \text{for $a \in A$ and $m \in M$}.
\]
Then similar to the above, we get that $M$ is a left $A_N$-module. Denote it by $M_{\NM}$. Moreover, we also have 
\begin{prop}
Let $(A,\N)$ be a Nijenhuis algebra and $(M,\NM)$ a left Nijenhuis module of $(A,\N)$. Then
\begin{enumerate}
\item For any $k \geq 0$, $(M,\NM^k)$ is a left $(A,\N^k)$-module.
\item For any $k, \ell \geq 0$, $(M_{\NM^k}, \NM^\ell)$ is a left $(A_{\N^k}, \N^{\ell})$-module.
\end{enumerate}
\end{prop}

\subsection{Free operated modules and free Nijenhuis modules}

We recall from~\mcite{Guop} that an {\bf operated algebra} is an algebra $A$ equipped with a linear operator $P:A\to A$.

\begin{defn}\mcite{QGG}
Let $(A, P)$ be an operated algebra.
\begin{enumerate}
\item A {\bf left operated $A$-module} is a pair $(M,P_M)$ consisting of a left $A$-module $M$ and a linear map $P_M:M \to M$.

\item Let $(M_1,P_{M_1})$ and $(M_2,P_{M_2})$ be two left operated $A$-modules. A {\bf left operated $A$-module homomorphism} from $(M_1,P_{M_1})$ to $(M_2,P_{M_2})$ is a left $A$-module homomorphism $f:M_1 \to M_2$ such that $f\circ P_{M_1} = P_{M_2} \circ f$.
\end{enumerate}
\end{defn}

For example, left Rota-Baxter modules in~\mcite{QGG} and left Nijenhuis modules are left operated modules. Now we recall the concept of free left operated modules and the construction of free left operated modules in~\mcite{QGG}.

\begin{defn}
Let $(A,P)$ be an operated algebra and X a set. A left operated $A$-module $(M(X), P_X)$ together with a set embedding $j_X: X \to M(X)$ is {\bf the free left operated $A$-module generated by $X$} if for any left operated $A$-module $(M,P_M)$ and any set map $f:X \to M$, there is a unique left operated $A$-module homomorphism $\overline{f}: M(X) \to M$ such that the following diagram commutes  
$$
\xymatrix{
X\ar@{^{(}->}[rr]^{j_X} \ar^{f}[drr]&&(M(X),P_X)\ar@{-->}^{\overline{f}}[d]&\\
&&(M,P_M). }
$$
\end{defn}

The free left operated $A$-module $(M(X),P_X)$ is constructed as follows~\mcite{QGG}. Denote
\begin{align*}
M(X):= AX \oplus (A\ot A)X \oplus \cdots = \oplus_{n\geq 1}(A^{\otimes n}X)= (\oplus_{n\geq 1}A^{\otimes n})X,
\end{align*}
where $A^{\ot n}X=A^{\ot n}\ot \bfk X$ with $\bfk X$ being the free $\bfk$-module on $X$. The left $A$-module structure of $M(X)$ is given by the action of $A$ on the left most tensor factor of $A^{\ot n}X$. Define a linear operator $P_X: M(X)\to M(X)$ by 
$$(a_1\ot \cdots \ot a_n)x\mapsto (1_A\ot a_1\ot \cdots \ot a_n) x, \ \text{for }  a_1, \cdots, a_n\in A, x\in X$$
and extend it by additivity. Define the map $j_X: X \to M(X), x \mapsto 1_A \ot x$, then by Proposition~2.10 in~\mcite{QGG}, we have
\begin{prop}
The pair $(M(X),P_X)$ with the map $j_X: X \to M(X)$ is the free left operated $(A,P)$-module generated by $X$. \mlabel{prop:freeo}
\end{prop}

Now we give the concept of free Nijenhuis modules and apply free operated modules to construct free Nijenhuis modules.

\begin{defn}
Let $(A,\N)$ be a Nijenhuis algebra and $X$ a set. A {\bf free left $(A,\N)$-module generated by $X$} is a left $(A,\N)$-module $(F(X),\N_{F(X)})$ together with a set map $j: X \to F(X)$
satisfying that for any left $(A,\N)$-module $(M,\N_M)$ and any set map $f: X \to M$, there exists a unique
left $(A,\N)$-module homomorphism $\widetilde{f}: F(X) \to M$ such that $\widetilde{f} \circ j= f$.
\end{defn}


Let $I_X$ denote the left operated submodule of $M(X)$ generated by the subset
$$
\{ \N(a)P_X(m) - P_X(\N(a)m+ aP_X(m)-P_X(am))~|~a\in A, m \in M(X)\}.
$$
Denote by $M(X)/I_X$ the quotient operated module of $M(X)$ by $I_X$ and define
$$\overline{P}_X: M(X)/I_X \to M(X)/I_X,\ m+ I_X \mapsto P_X(m)+I_X$$
to be the operator on the quotient $M(X)/I_X$ induced by $P_X$. For $a \in A$ and $m \in M(x)$, we have
\begin{align*}
& \N(a) \overline{P}_X( m+ I_X)=\N(a)(P_X(m)+I_X)=\N(a)P_X(m)+I_X\\
=& P_X(\N(a)m+ a P_X(m)-P_X(am))+I_X\\
=& P_X(\N(a)m+ I_X)+ P_X(a P_X(m)+I_X)-P_X(P_X(am)+I_X)\\
=& \overline{P}_X(\N(a)(m+I_X))+\overline{P}_X(a \overline{P}_X(m+I_X))-\overline{P}_X (\overline{P}_X(a(m+I_X)) ),
\end{align*}
hence $(M(X)/I_X, \overline{P}_X)$ is a left Nijenhuis $(A,\N)$-module.

\begin{theorem}
Let $(A,\N)$ be a Nijenhuis algebra and $X$ a set.
Then $(M(X)/I_X, \overline{P}_X)$ with the natural map $j:= \pi\circ j_X: X \to M(X) \to M(X)/I_X $ is the free left $(A,\N)$-module generated by $X$.
\mlabel{thm:freem}
\end{theorem}

\begin{proof}
We need to prove that for any left $(A,\N)$-module $(M,\NM)$ and any set map $f: X\to M$, there is a unique left $(A,\N)$-module homomorphism $\tilde{f}: (M(X)/I_X, \overline{P}_X) \to (M, \NM)$ such that $\tilde{f} \circ j=f$.

Note that $(A,\N)$ is also an operated algebra and $(M,\NM)$ is a left operated $A$-module. By Proposition~\mref{prop:freeo}, there is a unique left operated $A$-module homomorphism
$\overline{f}: (M(X),P_X) \to (M, \NM)$ such that $\overline{f} \circ j_X = f$,
as showing in the following diagram:
$$
\xymatrix{
X\ar^{j_X}[rr] \ar_{f}[drr] &&(M(X), P_X)\ar^{\pi}[rr]  \ar^{\overline{f}}[d]
&& (M(X)/I_X, \overline{P}_X)\ar@{-->}^{\tilde{f}}[dll] \\
&&(M,\NM).&&
}
$$
Now we show that $\overline{f}$ vanishes on the generators of $I_X$. For $a \in A$ and $m\in M(X)$, we have
\begin{align*}
& \overline{f}(\N(a)P_X(m) - P_X(\N(a)m+a P_X(m) -P_X(am)))\\
=& \N(a)\NM(\overline{f}(m)) - \NM(\N(a)\overline{f}(m)+a \NM(\overline{f}(m)) - \NM(a\overline{f}(m)))\\
=& 0.
\end{align*}
Since $\overline{f}$ vanishes on the generators of $I_X$ 
and $\overline{f}$ is a left operated $A$-module homomorphism, $\overline{f}(I_X)=0$. Thus $\overline{f}$ induces a unique left operated $A$-module homomorphism
$$\tilde{f}:(M(X)/I_X,\overline{P}_X) \to (M,\NM)$$
such that $\tilde{f}\circ \pi = \overline{f}$.

Since $(M(X)/I_X, \overline{P}_X)$ and $(M, \NM)$ are both left $(A,\N)$-modules and $\tilde{f} \circ \overline{P}_X= \NM \circ \tilde{f}$, $\tilde{f}$ is the unique left $(A,\N)$-module homomorphism such that
\begin{align*}
\tilde{f}\circ j = \overline{f}\circ \pi\circ j_X = \overline{f}\circ j_X = f,
\end{align*}
as required. 
\end{proof}

Analogous to the case of modules, we have
\begin{coro}
\begin{enumerate}
\item Every left Nijenhuis module is a quotient of a free left Nijenhui module.  \mlabel{it:quof}

\item Every finitely generated left Nijenhuis module is a quotient of a finitely generated free left Nijenhuis module. \mlabel{it:quff}
\end{enumerate}
\mlabel{coro:quof}
\end{coro}

\subsection{Free modules as free Nijenhuis modules} 
In this subsection, we give a restricted condition under which a free module is a free Nijenhuis module. Then we show that a Nijenhuis algebra is a free Nijenhuis module under the restricted condtion.

Let $(A,\N)$ be a Nijenhuis algebra and $(M,\NM)$ a left $(A,\N)$-module. Define the set of {\bf module constant} of $M$ by
\begin{align*}
MC(M):=\{m\in M \mid \NM(am)=\N(a)m, \, \text{ for all } a\in A\}.
\end{align*}

\begin{defn}
Let $X$ be a set. The {\bf restricted free left $(A,\N)$-module} generated by $X$  is a left $(A,\N)$-module $(F(X), \N_{F(X)})$ together with a set map $j_X : X \rightarrow F(X)$ satisfying that for any left $(A,\N)$-module $(M,\NM)$ and any set map $f:X \to M$ with  $\IIm(f) \subseteq MC(M)$, there exists a unique left $(A,\N)$-module homomorphism $\overline{f}: (F(X), \N_{F(X)}) \to (M,\NM)$ such that $\overline{f} \circ j_X=f$.
\end{defn}

Let $\rf(X)$ be the free left $\A$-module generated by $X$:
\begin{align*}
\rf(X):=\left\{\sum_{x\in X}a_{x}x~\big|~a_{x}\in A \right\} .
\end{align*}
Define a linear operator
\begin{align*}
\N_{\rf(X)}: \rf(X) \to \rf(X),\ \sum_{x\in X}a_{x}x \mapsto \sum_{x\in X}\N(a_{x})x.
\end{align*}

\begin{theorem} 
Let $(A,\N)$ be a Nijenhuis algebra and $X$ a set. Then
\begin{enumerate}
\item $(\rf(X),\N_{\rf(X)})$ is a left $(A,\N)$-module. \mlabel{it:rrbm}

\item   $(\rf(X),\N_{\rf(X)})$ together with the natural embedding map $j_X: X\to \rf(X)$,
is the  restricted free left $(A,\N)$-module generated by $X$.
\mlabel{it:frrbm}
\end{enumerate}
\mlabel{thm:fmrbm}
\end{theorem}

\begin{proof}
(\mref{it:rrbm}) We need to show that $\N_{\rf(X)}$ satisfies Eq.~(\mref{eq:lnm}). For any $a\in A$ and $a_x x \in \rf(X)$, we have
\begin{align*}
& \N(a)\N_{\rf(X)}( a_{x}x) = \N(a)( \N(a_{x})x)=\N(a)\N(a_{x})x \\
=&  \N(\N (a)a_{x})x +\N(a \N(a_{x}))x -\N (\N(aa_{x}))x \\
=& \N_{\rf(X)}(\N(a) (a_x x)) + \N_{\rf(X)}(a \N_{\rf(X)}(a_x x))-\N_{\rf(X)}(\N_{\rf(X)}(a(a_x x))),
\end{align*}
as required.

(\mref{it:frrbm}) 
For any left $(A,\N)$-module $(M,\NM)$ and any set map $f: X \to M$ with $\IIm(f) \subseteq MC(M)$, we show that there exists a unique left $(A,\N)$-module homomorphism $\overline{f}: (\rf(X), \N_{\rf(X)}) \to (M, \NM)$ such that $\overline{f} \circ j_X=f$.

By the universal property of free $A$-module $\rf(X)$, there is a unique left $A$-module homomorphism
\begin{align}
\overline{f}: \rf(X) \to M, \sum_{x\in X}a_{x}x \mapsto \sum_{x\in X}a_{x}f(x) \mlabel{eq:freemodule}
\end{align}
such that $\overline{f} \circ j_X=f$. Furthermore,
\begin{align*}
& (\overline{f} \circ \N_{\rf(X)}) (a_{x}x)  = \bar{f} (\N_{\rf(X)}( a_{x}x))\\
=& \overline{f} ( \N(a_{x})x)= \N(a_{x})f(x) \quad \text{(by Eq.~(\mref{eq:freemodule}))}\\
=& \NM(a_{x}f(x)) \quad \text{(by $\IIm(f) \subseteq MC(M)$)}\\
=& \NM(\overline{f}( a_{x}x)) \quad \text{(by Eq.~(\mref{eq:freemodule}))}\\
=& (\NM \circ \overline{f})( a_{x}x),
\end{align*}
hence $\overline{f}$ is the unique required left $(A,\N)$-module homomorphism.
\end{proof}

Taking $X=\{ x\}$ to be a singleton set, we know that the map sending $a$ to $ax$ is an isomorphism from $(A,\N)$ to $(\rf(X), \N_{\rf(X)}$ as left $(A,\N)$-modules. Hence

\begin{coro}
$(A,\N)$ is the restricted free left $(A,\N)$-module generated by a singleton set.
\end{coro}

\section{The ring of Nijenhuis operators and its characterizations}
\smallskip
\mlabel{sec:ringCon}

In this section, we introduce the notion of a ring of Nijenhuis operators and establish an equivalence between the modules of the ring of Nijenhuis operators and Nijenhuis modules. Moreover, the structure of the ring of Nijenhuis operators is characterized through a general construction.

\subsection{The ring of Nijenhuis operators and Nijenhuis modules} 
\mlabel{subsec:ringN}

Inspired by the ring of Rota-Baxter operators~\mcite{GL}, we define the ring of Nijenhuis operators as follows. 

\begin{defn}
Given a Nijenhuis algebra $(A,\N)$ and a polynomial algebra $ \bfk[Q]$ with variable $Q$, let $\bfk\langle A,\bfk[Q]\rangle$ be the free product (also called the coproduct) of $(A,\N)$ and $ \bfk[Q]$.  The {\bf ring of Nijenhuis operators on $(A,N)$}, denoted by $U_{\N}(A)$, is defined to be  the quotient
\[ U_{\N}(A)=\bfk\langle A,\bfk[Q]\rangle /I_{A,Q},\]
where $I_{A,Q}$ is the two-sided ideal of $\bfk\langle A, \bfk[Q]\rangle $ given by
\begin{equation}
I_{A,Q}=\langle QaQ-\N(a)Q+Q \N(a)-Q^2 a \, |\,  a\in A \rangle.
\mlabel{eq:NO}
\end{equation}

\end{defn}

Recall from~\mcite{GL} that an associative algebra $A$ equipped with a specific element $p \in A$ is called a {\bf pointed associative algebra}, denoted by $(A, p)$. A homomorphism $f: (A, p) \rightarrow (A', p')$ between such algebras is an algebra homomorphism $f: A \rightarrow A'$ that preserves the distinguished elements, i.e., $f(p) = p'$. Consequently, the pair $(U_{\N}(A), Q)$ forms a pointed associative algebra.

The definition of $U_{N}(A)$ is characterized by the following universal property. 

\begin{prop} \mlabel{prop:universal} Let $i:A\to \bfk\langle \A,\bfk[Q]\rangle$ be the natural embedding and $\pi:\bfk\langle \A,\bfk[Q]\rangle  \to U_{N}(\A)$ be the quotient homomorphism. For any pointed associative algebra $(A', p')$  and any algebra homomorphism $\phi: \A \rightarrow A'$ satisfying
\begin{equation}
\phi( \N(a))p'=p' \phi(a)p'+p' \phi( \N(a))-p'^2\phi(a), \, \text{ for all } a\in \A,
\mlabel{eq:Nuniv}
\end{equation}
there exists a unique pointed algebra homomorphism $\eta: (U_{N}(\A), Q) \rightarrow (A', p')$ such that $\phi = \eta \circ (\pi \circ i)$, which is obtained by an intermediate homomorphism $\xi: (\mathbf{k}\langle \A,\mathbf{k}[Q]\rangle, Q) \rightarrow (A', p')$ of  pointed algebras, as indicated in the following commutative diagram: 
\[\xymatrix{
\A \ar[r]^{i\quad \quad}\ar[dr]_{\phi} &\bfk\langle \A,\bfk[Q]\rangle\ar[r]^{\quad \pi}\ar[d]^{\xi}&U_N(\A)\ar[dl]^{\eta}\\
&(A',p')& 
}\]
\end{prop}

\begin{proof} 
Define an algebra homomorphism $\varphi: \bfk[Q] \rightarrow A'$ by sending $Q$ to $p'$. By the universal property of the free product $\bfk \langle A, \bfk[Q] \rangle$, there is a unique algebra homomorphism  $\xi:\bfk \langle \A, \bfk[Q]\rangle\rightarrow A'$ such that $\xi(a)=\phi(a)$ for all $a \in A$ and $\xi(Q)=p'$. Then
\begin{align*}
&\ \xi\Big(QaQ-\N(a)Q+Q\N(a)-Q^2 a\Big)\\
=&\ \xi(Q) \xi(a) \xi(Q)-\xi(\N(a)) \xi(Q)+\xi(Q) \xi(\N(a))-\xi(Q)^2 \xi(a) \\
=&\ p' \phi(a) p'-\phi(\N(a))p'+p'\phi(\N(a))-p'^2\phi(a)=0 \quad (\text{by Eq.~(\mref{eq:Nuniv}}). 
\end{align*}
This implies that the ideal $I_{\A, Q}$ is in the kernel of this algebra homomorphism $\xi:\bfk \langle \A, \bfk[Q]\rangle\rightarrow A'$, i.e., $\Ker \pi \subseteq \Ker \xi$. Hence we have a unique pointed algebra homomorphism $\eta:U_\N(\A) \to A'$ satisfying the property $\xi=\eta\circ \pi$. Consequently, the required property $\phi=\xi\circ i=\eta\circ(\pi \circ i)$ follows.
\end{proof}

\begin{prop}
Let $(A,N)$ be a Nijenhuis algebra, and let $U_N(A)$ be the ring of Nijenhuis operators on $(A,N).$ Then the structure of a left $(A, N)$-module is precisely equivalent to that of a left $U_N(A)$-
module extending the underlying $A$-module structure under the following correspondence.
\begin{enumerate}
    \item \mlabel{eq:ia1}
If $(M,\NM)$ is a left $(A,N)$-module, then $A$-module $M$ is also a left $U_N(A)$-module with the
action $Q\cdot m:=N_M(m), m\in M$.
    \item \mlabel{eq:ia2}
  If $M$ is a left $U_N(A)$-
module, then $(M,\NM)$ is a left $(A,N)$-module with the
operator $\NM:M\rightarrow M, m\mapsto Q\cdot m$.
\end{enumerate}
\end{prop}
\begin{proof}
(\mref{eq:ia1}) Given a left $(\A,\N)$-module $(M, \NM)$, we have a pointed algebra $(\End_{\bfk}(M), \NM)$. The action of $A$ on $M$ defines an algebra homomorphism $\phi: \A \to \End_{\bfk}(M)$. Moreover, by Eq.~(\mref{eq:lnm}) we know that Eq.~(\mref{eq:Nuniv}) holds where $p'=\NM$. Hence by Proposition~\mref{prop:universal}, there is a unique pointed algebra homomorphism $\eta:(U_{\N}(\A), Q) \to (\End_{\bfk}(M), \NM)$ with $\NM(m)=Q(m), m\in M$, giving $M$ a left $U_{\N}(\A)$-module structure.

(\mref{eq:ia2})  The left $U_{N}(\A)$-module structure of $M$ gives an algebra homomorphism $\eta: U_{\N}(\A) \to \End_{\bfk}(M)$. Also applying the restriction of $\eta$ to the subalgebra $\A \subseteq U_{\N}(\A)$ as showed in Lemma~\mref{lem:NUsum}, we obtain a left $\A$-module structure on $M$. Furthermore, by Eq.~\eqref{eq:Nuniv}, the linear map $\NM:= \eta(Q):M\rightarrow M$ defines a left $(\A,\N)$-module on $M$. 
 
\end{proof}  
Moreover, for left $(\A,\N)$-modules $(M, \NM)$ and $(M', \N_{M'})$, an $\A$-module homomorphism $ f: M\rightarrow M'$ is an $(\A,\N)$-module homomorphism if and only if $f$ is a $ U_{\N}(\A)$-module homomorphism. Thus we have the required isomorphism of categories as follows.    
\begin{theorem} 
The category of left $(\A,\N)$-modules is isomorphic to the category of left $U_{\N}(\A)$-modules.
\mlabel{thm:NM}
\end{theorem}
\delete{
\begin{exam} \Shilong{lack examples}We revisit the example at the end of Section~\mref{ss:decomp}. Let $ \bfk$ be any commutative ring and $\lambda\in \bfk$, then $ P=-\lambda : \bfk\rightarrow \bfk$ is a Rota-Baxter operator of weight $\lambda$. Then $U_{RB}(\bfk, P)=\bfk[t]/\langle t(t+\lambda)\rangle$. In fact, $U_{RB}(\bfk, P)$ is the Hecke algebra of $S_2$ over $\bfk$ with parameter $q=\lambda-1$.
\end{exam}
}

\begin{remark}
Note that $U_{\N}(\A)$ is a left module over itself. Then by Theorem~\mref{thm:NM}, $(U_{\N}(\A), \N_Q)$ is a left $(A,\N)$-module, where $\N_Q: U_{\N}(\A) \to U_{\N}(\A)$ is the left multiplcation by $Q$. Any left ideal $S$ of $U_{\N}(\A)$ is also a left $U_{\N}(A)$-module and hence $(S, \N_Q|_S)$ is a left $(\A,\N)$-module. We call $(S, \N_Q|_S)$ {\bf a left Nijenhuis ideal} of $(U_{\N}(\A), \N_Q)$.  
\end{remark}

\subsection{General construction of the ring of Nijenhuis operators}
\mlabel{sec:Nst}

Theorem~\mref{thm:NM} shows that the study of Nijenhui modules is governed by the ring of Nijenhuis operators. In this subsection, we give a general construction of the ring of Nijenhuis operators.

Proposition~\ref{prop:universal} describes the universal property of the ring of Nijenhuis operators. As an application, we have
\begin{lemma}
Let $(\A,\N)$ be a  Nijenhuis algebra. Then the ring $U_\N(\A)$ of Nijenhuis operators on $(\A,\N)$  can be regarded as a direct sum of $\A$-bimodules 
\begin{equation}
U_N(A)=A\oplus \langle Q\rangle, 
\mlabel{eq:Us}
\end{equation} 
where $\langle Q \rangle$ is the ideal generated by $Q$ in $U_{\N}(\A)$. 
\mlabel{lem:NUsum}
\end{lemma}
\begin{proof}
As a special case of Proposition~\ref{prop:universal},
taking $(A',p')=(A,0)$ and $\phi=\id_\A $, an algebra homomorphism $\eta: U_{\N}(\A)\rightarrow A$ follows where $ \id_\A=\eta\circ (\pi \circ i)$.  Hence, $\pi \circ i$ is injective and we can regard $\A$ as a subalgebra of $U_{\N}(\A)$. Note that $U_{\N}(\A) / \langle Q \rangle \cong \A$ which gives a short exact sequence $$0\rightarrow \langle Q \rangle\rightarrow U_{\N}(\A)\xrightarrow{\eta} \A\rightarrow 0.$$ Also by $ \id_\A=\eta\circ (\pi \circ i)$, the  short exact sequence splits and then we have a direct sum decomposition $U_{\N}(\A) = \A \oplus \langle Q \rangle$ as $\A$-bimodules.     
\end{proof}

The definition of $I_{\A,Q}$ in Eq.~(\mref{eq:NO}) provides a relation in $U_\N(\A)$: 
\begin{equation}
QaQ =N(a)Q-QN(a) +Q^2a.  
\mlabel{eq:NOi}
\end{equation} 

When $a=1_\A$ in Eq.~(\ref{eq:NOi}), we obtain $Q1_\A Q= Q^21_\A= Q^2$ and then $-N(1_\A)Q+QN(1_\A)=0$.

Let $uQaQv$ be a monomial in the ideal $\langle Q \rangle$, where $u$ and $v$ are monomials in $U_\N(\A)$, and $a\in \A$. The subsequence $QaQ$ is called a {\bf gap} in the monomial $uQaQv$. Here $a$ could be $1_A$, i.e., $uQaQv=uQ^2v$. Further, when $a\neq 1_A$, we call $QaQ$ a {\bf strict gap}. 

In the subsequent contents, unless otherwise specified, all gaps we consider are strict. The following lemma shows that the gaps can be eliminated by Eq.~(\ref{eq:NOi}). 
\begin{lemma}
In the ring of Nijenhuis operators $U_{\N}(\A)$, we have
\begin{enumerate} 
\item \mlabel{eq:ib1}
$Q^naQ=N^n(a)Q-QN^n(a)+Q^{n+1}a.$
\mlabel{it:n1}
\item \mlabel{eq:ib2}
$Q^naQ^m=N^n(a)Q^m-Q^mN^n(a)+Q^{n+m}a.$
\mlabel{it:n2}
\end{enumerate}
\mlabel{lem:cpgap}
\end{lemma}
\begin{proof}
(\mref{eq:ib1}) We carry out the verification by induction on $n$.

When $n=1$, the result follows directly from Eq.~(\ref{eq:NOi}). Assume the result holds for $n = k$, i.e., 
\begin{align*}
Q^kaQ=N^k(a)Q-QN^k(a)+Q^{k+1}a.
\end{align*}
Then for the case $n = k+1$, we have
\begin{align*}
&\ Q^{k+1}aQ=Q(Q^{k}aQ)=QN^k(a)Q-Q^2N^k(a)+Q^{k+2}a\quad \text{(by induction hypothesis)}\\
=&\ N^{k+1}(a)Q-QN^{k+1}(a)+Q^2N^k(a)-Q^2N^k(a)+Q^{k+2}a\quad \text{(by Eq.~(\mref{eq:NOi}))}\\
=&\ N^{k+1}(a)Q-QN^{k+1}(a) +Q^{k+2}a,  
\end{align*} 
as required.

(\mref{eq:ib2}) The result is proven by induction on $m$, with part (\ref{it:n1}) serving as the base case $m=1$.

Assume the result holds for $m = k$, i.e.,
\begin{align*}
Q^naQ^k=N^n(a)Q^k-Q^kN^n(a)+Q^{n+k}a.
\end{align*}
For the case $m = k+1$, we have
\begin{align*}
&\ Q^{n}aQ^{k+1}= (Q^{n}aQ^{k})Q= N^n(a)Q^{k+1}-Q^kN^n(a)Q+Q^{n+k}aQ\quad \text{(by induction hypothesis)}\\ 
=&\N^n(a)Q^{k+1}-\Big(N^{n+k}(a)Q-QN^{n+k}(a)+Q^{k+1}N^n(a)\Big)\\
&\ +\Big(N^{n+k}(a)Q-QN^{n+k}(a)+Q^{k+n+1} a \Big)\quad \text{(by ~part (\ref{it:n1})
)}\\
=&\ N^n(a)Q^{k+1}- Q^{k+1}N^n(a) +Q^{k+n+1} a\\
=&\ N^n(a)Q^{m}- Q^{m}N^n(a) +Q^{m+n } a.
\end{align*} 
This completes the inductive step for part (\ref{it:n2}).  

\end{proof} 

By repeatedly applying Eq.~(\mref{eq:NOi}), we obtain the following direct sum decomposition of the two-sided ideal$\langle Q\rangle$.  
\begin{lemma}  
The ideal
$\langle Q\rangle$ has a direct sum decomposition into $\A$-bimodules: \begin{equation}
\langle Q\rangle =\mathop{\oplus}\limits_{i=1}^\infty AQ^iA:= AQA\oplus AQ^2A\oplus AQ^3A\oplus\cdots. 
\mlabel{eq:sumQ}
\end{equation} 
\end{lemma}
\begin{proof} 
It is clear that $\langle Q\rangle \supseteq \mathop{\oplus}\limits_{i=1}^\infty AQ^iA$. It remains to show that by iterative application of the relation ~(\ref{eq:NOi}), every monomial $\omega \in \langle Q\rangle$ can be expressed as a linear combination of monomials in $\mathop{\oplus}\limits_{i=1}^\infty AQ^i A$.

We proceed by induction on the number $n$ of  gaps in $\omega$.

For the base case $n = 0$, we have $\omega \in AQA \subseteq \mathop{\oplus}\limits_{i=1}^\infty AQ^iA$.

Now assume $n > 0$ and that the claim holds for all monomials with fewer than $n$  gaps. Let $\omega = sQaQv$, where $QaQ$ is the first  gap in $\omega$. Applying the relation ~(\ref{eq:NOi}), we obtain 
$$sQaQv=s(N(a)Q)v-s(QN(a))v+s(Q^2a)v.$$
Each term on the right-hand side has fewer than $n$  gaps. By the induction hypothesis, each term belongs to $\mathop{\oplus}\limits_{i=1}^\infty AQ^iA$, and hence so does $\omega$. This completes the induction. 
\end{proof}

\begin{theorem} 
Let $(A,N)$ be a Nijenhuis algebra. Then  \begin{equation}
U_N(A)= A\oplus\Big(\mathop{\oplus}\limits_{i=1}^\infty AQ^iA \Big),
\mlabel{eq:sumU}
\end{equation} 
where for arbitrary $a\in A,$ $a_1Q^na_2\in AQ^nA$ and $a_3Q^ma_4\in AQ^mA,$  the multiplication is given by $a\cdot(a_1Q^na_2)=aa_1Q^na_2, (a_1Q^na_2)\cdot a=a_1Q^n a_2a$ and
\begin{equation}
(a_1Q^na_2)\cdot (a_3Q^ma_4)= 
a_1 N^n(a_2 a_3)Q^{m}a_4- a_1Q^{m}N^n(a_2 a_3)a_4 +a_1Q^{m+n } a_2 a_3a_4.
\mlabel{eq:Qmu}
\end{equation} 
\end{theorem}
\begin{proof}
By Eqs.~(\ref{eq:Us}) and (\ref{eq:sumQ}), we obtain the direct sum decomposition in Eq.~(\ref{eq:sumU}).
The multiplication  follows directly from the definition of the ring of Nijenhuis operators and Lemma~\ref{lem:cpgap}(\ref{it:n2}).    
\end{proof}

The $\A$-bimodule $AQ^nA$ is generated by $Q$ and $A \ot A$ is generated by $1_A \ot 1_A$ as an $A$-bimodule. This induces an isomorphism of $A$-bimodules: $$AQ^nA\rightarrow A\otimes A, a_1Q^na_2\mapsto a_1\otimes a_2.$$
Let $A \otimes_1 A, A \otimes_2 A, A \otimes_3 A, \ldots$ denote the same space $A \otimes A$, but with indices to indicate the isomorphism with $AQA, AQ^2A, AQ^3A, \ldots$, respectively. The following result is then obtained.
\begin{coro} 
Let $(A,N)$ be a Nijenhuis algebra. Then   
we have an algebra isomorphism 
$$U_N(A)\cong A\oplus (A \otimes_1 A)\oplus (A \otimes_2 A)\oplus (A \otimes_3 A)\oplus \cdots$$
where the multiplication of the latter algebra  is given as follows:$a\cdot(a_1\otimes_n a_2)=aa_1\otimes_n a_2, (a_1\otimes_n a_2)\cdot a=a_1\otimes_n  a_2a$ and
\begin{equation*}
(a_1\otimes_n a_2)\cdot (a_3\otimes_m a_4)= 
a_1 N^n(a_2 a_3)\otimes_m a_4- a_1\otimes_m N^n(a_2 a_3)a_4 +a_1\otimes_{m+n} a_2 a_3a_4. 
\end{equation*} 
\end{coro}

\section{Projective and injective Nijenhuis modules}
\mlabel{sec:pi}

In this section, we introduce the concept of projective left $(A,\N)$-modules and injective left $(A,\N)$-modules. Then we show that there are enough projective and injective objects in the category of left $(A,\N)$-modules.

\begin{defn}
Let $(A,\N)$ be a Nijenhuis algebra. A left $(A,\N)$-module $(P, \N_P)$ is {\bf projective} if for every left $(A,\N)$-module epimorphism $f:(M,\NM)\to (M',\N_{M'})$ and every left $(A,\N)$-module homomorphism $g:(P,\N_P)\to (M',\N_{M'})$, there exists a left $(A,\N)$-module homomorphism $\overline{g}:(P,\N_P)\to (M,\NM)$ such that the following diagram commutes:
$$
\xymatrix{
&(P,\N_P)\ar^{g}[d] \ar@{-->}_{\overline{g}}[dl]&\\
(M,\NM)\ar_{f}[r]&(M',\N_{M'})\ar[r]&0.}
$$
\mlabel{defn:proj}
\end{defn}

Then analogous to that every free module is a projective module, we have the following result.
\begin{prop}
Every free left $(A,\N)$-module is a projective left $(A,\N)$-module.
\mlabel{prop:frep}
\end{prop}

\begin{proof}
Let $(F(X), \N_{F(X)})$ be the free left $(A,\N)$-module generated by $X$ with the natural embedding $j: X \rightarrow F(X)$. Let $f:(M,\NM)\to (M', \N_{M'})$ be a left $(A,\N)$-module epimorphism and let $g: (F(X),\N_{F(X)}) \to (M',\N_{M'})$ be a left $(A,\N)$-module homomorphism.

Since $f$ is surjective, for each $x\in X$,
there is a $m_x\in M$ such that $f(m_x) = g(x)$. Define a set map $g_0: X \to M, x \mapsto m_x$. Then by the universal property of $(F(X),\N_{F(X)})$, there is a left $(A, \N)$-module homomorphism $\overline{g}:F(X) \to N$ such that $ \overline{g}\circ j = g_0$. Hence $ f\circ \overline{g}\circ j = f\circ g_0$. By the definition of $g_0$, we also have $g \circ j=f \circ g_0$. By the universal property of $(F(X), \N_{F(X)})$, there should be a unique left $(A, \N)$-module homomorphism $\tilde{g}: F(X) \rightarrow M'$ such that $\tilde{g} \circ j=f \circ g_0$. Hence $g=\tilde{g}=f \circ \overline{g}$. This completes the proof.

\end{proof}

\begin{remark}
By Corollary~\mref{coro:quof} and Proposition~\mref{prop:frep}, there are enough projective objects in the category of Nijenhuis modules and hence every Nijenhuis module has a projective resolution.
\end{remark}

Now we introduce the concept of injective $(A,\N)$-modules and then we show that there are enough injective $(A,\N)$-modules.

\begin{defn}
Let $(A,\N)$ be a Nijenhuis algebra. A left $(A,\N)$-module $(I,\N_I)$ is {\bf injective} if for every left $(A,\N)$-module monomorphism $f:(M,\NM) \to (M', \N_{M'})$ and every left $(A,\N)$-module homomorphism $g:(M,\NM) \to (I,\N_I)$,
there exists a left $(A,\N)$-module homomorphism $\overline{g}: (M',\N_{M'}) \to (I,\N_I)$ such that the following diagram commutes:
$$
\xymatrix{
&(I,\N_I)\ar@{<--}^{\overline{g}}[dr]&\\
0\ar[r]&(M,\NM)\ar_{f}[r]\ar^{g}[u]&(M',\N_{M'}).}
$$
\end{defn}

By Theorem~\mref{thm:NM}, the category of left $(\A,\N)$-modules is isomorphic to the category of left $U_{\N}(\A)$-modules. Now we use this result to give the Baer Criterion for injective Nijenhuis modules.

\begin{prop}\mlabel{prop:baercri}
Let $(I,\N_I)$ be a left $(\A,\N)$-module. Then $(I,\N_I)$ is an injective left $(\A,\N)$-module if and only for every left Nijenhui ideal $(S,\N_Q|_S)$ of $(U_{\N}(\A),\N_{Q})$, every $(\A,\N)$-module homomorphism $h:(S,\N_Q|_{S})\rightarrow (I,\N_I)$ can be extended to a homomorphism from $(U_{\N}(\A), \N_Q)$.
\end{prop}
\begin{proof}
We adapt the proof of the Baer Criterion as presented for example in~\mcite{Rot}.

Assume that $(I,\N_I)$ is an injective $(\A,\N)$-module. Then by the definition of injective Nijenhuis modules, every $(\A,\N)$-module homomorphism $h:(S,\N_Q|_{S}) \rightarrow (I,\N_I)$ can be extended to one from $(U_{\N}(\A),\N_Q)$.

Conversely, assume that for every left Nijenhui ideal $(S,\N_Q|_S)$ of $(U_{\N}(\A),\N_{Q})$, every $(\A,\N)$-module homomorphism $h:(S,\N_Q|_{S})\rightarrow (I,\N_I)$ can be extended to one from $(U_{\N}(\A), \N_Q)$. Let $f: (M,\NM)\to (M',\N_{M'})$ be a left $(\A,\N)$-module monomorphism and let $g: (M,\NM)\to (I,\N_I)$ be a left $(\A,\N)$-module homomorphism. We need to prove that there is a left $(\A,\N)$-module homomorphism $\overline{g}: (M',\N_{M'}) \to (I, \N_I)$ such that $g=\overline{g} \circ f$.

Identify $(M,\NM)$ as a left $(\A,\N)$-submodule of $(M',\N_{M'})$ and consider all triples $(H,\N_H, h)$ consisting of a left $(\A,\N)$-module $(H,\N_H)$ with $(M, \NM) \leq (H,\N_H) \leq (M',\N_{M'})$ and a left $(A,\N)$-module homomorphism $h: (H,\N_H) \rightarrow (I,\N_I)$ with $h|_{M}=g$. Define a partial order on these triples by $(H,\N_H, h) \leq (H',\N_{H'}, h')$ when $(H,\N_H) \leq (H',\N_{H'})$ and $h'|_{H}=h$. Then for any linearly ordered collection $(H_i, \N_{H_i},h_i)$ of these triples, there is an upper bound $(\overline{H}, \N_{\overline{H}}, \overline{h})$ with $(H_i, \N_{H_i}, h_i) \leq (\overline{H}, \N_{\overline{H}}, \overline{h})$ by taking $\overline{H}$ to be the union of all $H_i$ with $\N_{\overline{H}}(m)=\N_{H_i}(m)$ and $\overline{h}(m)=h_i(m)$ when $m \in H_i$. Hence, by Zorn's lemma, there is a maximal such triple $(H_{\infty}, \N_{H_{\infty}}, h_{\infty})$. If $H_{\infty}=M'$, then we are done. If not, there is an element $b \in M' \setminus H_{\infty}$. View $H_{\infty}$ as a left $\NR$-module by Theorem~\mref{thm:NM} and define
\[
L :=\{ r \in \NR \, \mid \, rb \in H_{\infty}\}.
\]
Then $L$ is a left ideal of $\NR$ and $(L, \N_Q|_L)$ is a left Nijenhuis ideal of $(\NR,\N_Q)$. Define a map 
\[
\alpha: L \rightarrow I, \, r \mapsto h_{\infty}(rb).
\]
Then $\alpha$ is a left $\A$-module homomorphism and for $r \in L$,
\[
\alpha(\N_Q|_L(r))=\alpha(Qr)=h_{\infty}(Qrb)=h_{\infty}(\N_{H_{\infty}}(rb))=\N_I(h_{\infty}(rb))=\N_I(\alpha(r)),
\]
so $\alpha$ is a left $(\A,\N)$-module homomorphism. By assumption, there is a left $(\A,\N)$-module homomorphism $\alpha^{\ast}: (\NR, \N_Q) \rightarrow I$ extending $\alpha$.
Define $H_0=H_{\infty}+ \NR b$ and define a map
\[
h_0: H_0 \rightarrow I, \, a+rb \mapsto h_{\infty}(a)+r \alpha^{\ast}(1_{\NR}), 
\]
where $a \in H_{\infty}$ and $r \in \NR$. If $a+rb=a'+r'b$ with $a,a' \in H_{\infty}$ and $r,r' \in \NR$, then $(r'-r)b=a-a' \in H_{\infty}$ and hence $r'-r \in L$. Therefore,
\[
h_{\infty}(a-a')=h_{\infty}((r'-r)b)=\alpha(r'-r)=\alpha^{\ast}(r'-r)=(r'-r) \alpha^{\ast} (1_{\NR}),
\]
which means $h_0$ is well-defined. Moreover, $h_0$ is a left $\NR$-module homomorphism and hence a left $(\A,\N)$-module homomorphism by Theorem~\mref{thm:NM}. By the definition of $H_0$ and $h_0$, $(H_{\infty}, \N_{H_{\infty}}, h_{\infty}) <(H_0, \N_{H_0}, h_0)$, contradicting the maximality of $(H_{\infty}, \N_{H_{\infty}}, h_{\infty})$. Therefore, $H_{\infty}=M'$.
\end{proof}

\begin{lemma}\mlabel{lem:isnijm}
Let $(\A,\N)$ be a Nijenhuis algebra and $G$ an abelian group. Then $(\Hom_{\mathbb{Z}}(\NR, G), \N_{\ell})$ is a left $(\A,\N)$-module, where
\begin{align*}
\N_{\ell}:\Hom_{\mathbb{Z}}(\NR, G) \rightarrow \Hom_{\mathbb{Z}}(\NR, G),\, f \mapsto \N_{\ell}(f), 
\end{align*}
with $(\N_{\ell}(f))(x)=f(xQ)$ for all $x \in \NR$.
\end{lemma}

\begin{proof}
Since $\NR$ is a right $\NR$-module, $\Hom_{\mathbb{Z}}(\NR, G)$ is a left $\NR$-module via $(af)(x):=f(xa)$ for $f \in \Hom_{\mathbb{Z}}(\NR, G)$, $a \in \A$ and $x \in \NR$. By Lemma~\mref{lem:NUsum}, $A$ is a subalgebra of $\NR$, $\Hom_{\mathbb{Z}}(\NR, G)$ is also a left $\A$-module. 

Moreover, for any $f \in \Hom_{\mathbb{Z}}(\NR, G)$, $a \in A$ and $x \in \NR$,
\begin{align*}
& \big( \N_{\ell}(\N(a)f)+\N_{\ell}(a \N_{\ell}(f))-\N_{\ell}^2(af) \big)(x)\\
=& f \big(x Q \N(a)+x QaQ-xQ^2a \big)\\
=& f(x \N(a) Q) \quad \text{(by the definition of $\NR$)}\\
=& (\N_{\ell}(f))(x \N(a))=(\N(a) \N_{\ell}(f))(x),
\end{align*}
as required.
\end{proof}

Recall that an abelian group $D$ is {\bf divisible}, if for each $d\in D$ and every nonzero integer $n$, there exists $d' \in G$ such that $d=nd'$.

\begin{prop}\mlabel{prop:injnijm}
Let $(\A,\N)$ be a Nijenhuis algebra and $D$ a divisible abelian group. Then $(\Hom_{\mathbb{Z}}(\NR,D),\N_{\ell})$ is an injective left $(\A,\N)$-module.
\end{prop}

\begin{proof}
By Lemma~\mref{lem:isnijm}, $(\Hom_{\mathbb{Z}}(\NR,D),\N_{\ell})$ is a left $(\A,\N)$-module. Now we show that for any Nijenhuis ideal $(S, \N_Q|_S)$ of $(\NR, \N_Q)$ and any left $(A,\N)$-module homomorphis $h:(S, \N_Q|_S) \rightarrow (\Hom_{\mathbb{Z}}(\NR,D),\N_{\ell})$, there is a left $(\A,\N)$-module homomorphism $\overline{h}: (\NR, \N_Q) \rightarrow (\Hom_{\mathbb{Z}}(\NR,D),\N_{\ell})$ such that the following diagram commutes: 
$$
\xymatrix{
&(\Hom_{\mathbb{Z}}(\NR,\N_Q), \N_{\ell})\ar@{<--}^{\overline{h}}[dr]&\\
0\ar[r]&(S, \N_Q|_{S})\ar_{i}[r]\ar^{h}[u]&(\NR, \N_Q).}
$$

Define a map $\phi: S\rightarrow D$ by $\phi(s)=h(s)(1_{\NR})$. Then $\phi$ is a $\mathbb{Z}$-module homomorphism.  Since an abelian group is an injective $\mathbb{Z}$-module if and only if it is divisible~\mcite{Rot}, $D$ is an injective $\mathbb{Z}$-module. Then there is a $\mathbb{Z}$-module homomorphism $\psi: \NR \rightarrow D$ such that $\phi=\psi \circ i$.

Then define a map $\overline{h}: (\NR, \N_Q) \rightarrow (\Hom_{\mathbb{Z}}(\NR,D),\N_{\ell})$ by
\[
\overline{h}(x)(y)= \psi(yx), \, \text{for $x, y \in \NR$.}
\]
For $a \in \A$, we have 
\[
\overline{h}(ax)(y)=\psi(y(ax))=\psi((ya)x)=\overline{h}(x)(ya)=(a \overline{h}(x))(y),
\]
hence $\overline{h}$ is an $\A$-module homomorphism. Moreover,
\[
\big( \overline{h} (\N_Q(x)) \big)(y)=\overline{h}(Qx)(y)=\psi(yQx)=(\overline{h}(x))(yQ)=\N_{\ell}(\overline{h}(x))(y),
\]
hence $\overline{h}: (\NR, \N_Q) \rightarrow (\Hom_{\mathbb{Z}}(\NR,D),\N_Q)$ is an $(\A,\N)$-module homomorphism.

Let $s \in S$. For $x=a \in \A$, we have
\begin{align*}
\big( (\overline{h} \circ i)(s) \big)(a)=(\overline{h}(s))(a)=\psi(as)=\phi(as)=(h(as))(1_{\NR})=(ah(s))(1_{\NR})=(h(s))(a),
\end{align*}
and for $x=Q$, we have 
\begin{align*}
& \big((\overline{h} \circ i) (s)\big) (Q)=(\overline{h}(s))(Q)=\psi(Qs)=\phi(Qs)=(h(Qs))(1_{\NR})\\
=& (h(\N_Q(s)))(1_{\NR})=(\N_{\ell}(h(s)))(1_{\NR})=h(s)(Q),
\end{align*}
which implies $\overline{h} \circ i=h$. Thus $(\Hom_{\mathbb{Z}}(\NR,D),\N_{\ell})$ is an injective left $(\A,\N)$-module by Proposition~\mref{prop:baercri}
\end{proof}

\begin{theorem}
Let $(\A,\N)$ be a Nijenhuis algebra and $(M,\NM)$ a left $(\A,\N)$-module. Then $(M,\NM)$ can be embedded into an injective left $(\A,\N)$-module.
\mlabel{thm:sinj}
\end{theorem}

\begin{proof}
By Theorem~\mref{thm:NM}, $M$ is a left $\NR$-module. For each $m \in M$, define a map 
\[
\varphi_m: \NR \rightarrow M, \, x \mapsto xm. 
\]
Then $\varphi_m \in \Hom_{\mathbb{Z}}(\NR, M)$ and define 
\[
f: (M,\NM) \rightarrow (\Hom_{\mathbb{Z}}(\NR, M), \N_{\ell}), \, m \mapsto \varphi_m.
\]
For any $a \in \A$, $x\in \NR$ and $m\in M$, we have
\begin{align*}
f(am)(x)= \varphi_{am}(x)=x(am)=(xa)m=\varphi_{m}(xa)=(a\varphi_{m})(x)=(af(m))(x),
\end{align*}
so $f$ is an $\A$-module homomorphism. Moreover,
\begin{align*}
& (f(\NM(m)))(x)=\varphi_{\NM(m)}(x)=x \NM(m)=x(Qm)=(xQ)m\\
=& \varphi_{m}(xQ)=(\N_{\ell}(\varphi_m))(x)=(\N_{\ell}(f(m)))(x),
\end{align*}
hence $f \circ \NM=\N_{\ell} \circ f$ and $f$ is an $(\A,\N)$-module homomorphism. 

Now we show that $f$ is a monomorphism. For any $m_1, m_2\in M$, if $f(m_1)=f(m_2)$, i.e. $\varphi_{m_1}=\varphi_{m_2}$, then $xm_1=\varphi_{m_1}(x)=\varphi_{m_2}(x)=xm_2$ for all $x\in \NR$. In particular, taking $x =1_{\NR}$ yields $m_1=m_2$.

Since every abelian group can be embedded into a divisible abelian group~\mcite{Rot}, there exists an embedding map $i_M: M \rightarrow D$ with $D$ a divisible abelian group.
Then we have 
\[(i_M)_{\ast}: (\Hom_{\mathbb{Z}}(\NR,M), \N_{\ell}) \rightarrow (\Hom_{\mathbb{Z}}(\NR,D), \N_{\ell}), \, \varphi \mapsto i_M \circ \varphi.
\]
For $a \in A$ and $x \in \NR$, we have 
\begin{align*}
& ((i_M)_{\ast}(a \varphi))(x)=(i_M \circ (a \varphi))(x)=i_M((a\varphi)(x))\\
=& i_M(\varphi(xa))=(i_M \circ \varphi)(xa)=(a(i_M \circ \varphi))(x)=(a((i_M)_{\ast}(\varphi)))(x),
\end{align*}
so $(i_M)_{\ast}(a \varphi)=a((i_M)_{\ast}(\varphi))$ and $(i_M)_{\ast}$ is an $\A$-module homomorphism. Moreover,
\begin{align*}
& ((i_M)_{\ast}(\N_{\ell}(\varphi)))(x)=(i_M \circ (\N_{\ell}(\varphi)))(x)=i_M((\N_{\ell}(\varphi))(x))=i_M(\varphi(xQ))\\
=& (i_M \circ \varphi)(xQ)=((i_M)_{\ast}(\varphi))(xQ)=(\N_{\ell}((i_M)_{\ast}(\varphi)))(x),
\end{align*}
so $(i_M)_{\ast}(\N_{\ell}(\varphi))=\N_{\ell}((i_M)_{\ast}(\varphi))$ and $(i_M)_{\ast}$ is an $(\A,\N)$-module homomorphism. Then $(i_M)_{\ast} \circ f: (M, \NM) \rightarrow (\Hom_{\mathbb{Z}}(\NR,D), \N_{\ell})$ is an $(\A,\N)$-module homomorphism and by Proposition~\mref{prop:injnijm}, $(\Hom_{\mathbb{Z}}(\NR,D), \N_{\ell})$ is an injective $(\A,\N)$-module. This completes the proof.
\end{proof}

\section{Flat Nijenhuis modules}
\mlabel{sec:flat}

In this section, we introduce the tensor product of Nijenhuis modules and then we study flat Nijenhuis modules.

\subsection{Tensor product of Nijenhuis modules}
To define the tensor product of Nijenhuis modules, we first introduce the concept of right Nijenhuis modules.

\begin{defn}
Let $(A,\N)$ be a Nijenhuis algebra. A {\bf right (Nijenhuis) $(A,\N)$-module} $(M,\NM)$ is a
right $A$-module $M$ together with a linear operator $\NM: M \to M$ such that
\begin{align}
\NM(m) \N(a)=\NM( \NM(m)a + m\N(a)-\NM (ma)), \, \text{for any } a\in A, m\in M. \mlabel{eq:rnm}
\end{align}
\end{defn}

The definition of right $(A,\N)$-module homomorphisms is analogous to that of left $(A,\N)$-module homomorphisms and is omitted here.
Denote by ${\bf Mod}_{(A,\N)}$ the category of right $(A,\N)$-modules.

\begin{defn}
Let $(\A,\N)$ be a Nijenhuis algebra, $(M_1,\N_{M_1})$  a right $(\A,\N)$-module and $(M_2,\N_{M_2})$ a left $(\A,\N)$-module.
\begin{enumerate}
\item
Let $G$ be an (additive) abelian group. A map $f: M_1 \times M_2 \rightarrow G$ is called $(\A,\N)$-{\bf bilinear} if for all $m_1, m'_1\in M$, $m_2, m'_2\in M_2$ and $a \in \A$, we have
\begin{align*}
f(m_1+m'_1, m_2)&=f(m_1, m_2)+f(m'_1, m_2),\\
f(m_1, m_2+m'_2)&=f(m_1, m_2)+f(m_1, m'_2),\\
f(m_1 a, m_2)&=f(m_1, am_2),\\
f(\N_{M_1}(m_1), m_2)&=f(m_1, \N_{M_2}(m_2)).
\end{align*}
\item
The ${\bf tensor~product}$ of  $(M_1,\N_{M_1})_{(\A,\N)}$ and $(_{(\A,\N)}M_2,\N_{M_2})$ over $(\A,\N)$, denoted by $M_1 \ot_{(\A,\N)} M_2$, is an abelian group together with
an $(\A,\N)$-bilinear map
$$ \iota: M_1 \times M_2 \rightarrow M_1\ot_{(\A,\N)} M_2$$ satisfying that for every abelian group $G$ and every $(\A,\N)$-bilinear map $f: M_1 \times M_2 \rightarrow G$,
there exists a unique abelian group homomorphism $\widetilde{f}: M_1 \otimes_{(\A,\N)} M_2 \rightarrow G$ such that the following diagram commutes
$$
\xymatrix{
&M_1\times M_2 \ar^{\iota}[rr] \ar_{f}[dr]&&M_1 \otimes_{(\A,\N)} M_2.\ar@{-->}^{\widetilde{f}}[dl]\\
&&G&}
$$
\end{enumerate}
\mlabel{def:rbmt}
\end{defn}

The following result gives a construction of the tensor product of Nijenhuis modules.

\begin{theorem}
Let $(\A,\N)$ be a Nijenhuis algebra and let $(M_1,\N_{M_1})$ be a right $(A,\N)$-module, $(M_2,\N_{M_2})$ a left $(A, \N)$-module.
Let $F$ be the free abelian group on the set $M_1 \times M_2$ and $I$ the subgroup of $F$ generated by  elements of the form
\begin{align*}
& (m_1+m'_1, m_2) - (m_1, m_2)-(m'_1, m_2),\\
& (m_1, m_2+m'_2)- (m_1, m_2) - (m_1, m'_2),\\
& (m_1a, m_2)- (m_1, am_2), \\ 
& (\N_{M_1}(m_1), m_2)-(m_1, \N_{M_2}(m_2)),
\end{align*}
for all $m_1, m'_1\in M_1, m_2, m'_2\in M_2$ and $a\in \A$.
Then $F/I$ with the natural map $\iota: M_1 \times M_2 \rightarrow F \rightarrow F/I$ is the tensor product of $(M_1,\N_{M_1})$ and $(M_2,\N_{M_2})$ over $(\A,\N)$
\mlabel{thm:rbmt}
\end{theorem}

\begin{proof}
For $(m_1,m_2)\in M_1\times M_2$, denote the pure tensor by $m_1 \ot_{(\A,\N)} m_2:= \iota((m_1,m_2))$. Then every element of $F/I$ is a finite sum of pure tensors. 

We need to prove that for any abelian group $G$ and for any $(\A,\N)$-bilinear map $f: M_1\times M_2 \rightarrow G$, there is a unique abelian group homomorphism $\tilde{f}: F/I \to G$ such that $f=\tilde{f} \circ \iota$.

Since $F$ is the free abelian group on the set $M_1 \times M_2$, the map $f$ extends uniquely to an abelian group homomorphism $\overline{f}: F \to G$. By $f$ being $(A,\N)$-bilinear, $\overline{f}$ vanishes on the generators of $I$, hence $\overline{f}$ induces a unique abelian group homomorphism $\tilde{f}: F/I \to G$. For $m_1 \in M_1$ and $m_2 \in M_2$, we have 
\[
(\tilde{f} \circ \iota)((m_1,m_2))=\tilde{f} (m_1 \ot_{(\A,\N)} m_2)=\overline{f}((m_1,m_2))=f((m_1,m_2)),
\]
so $\tilde{f}: F/I \to G$ is the unique abelian group homomorphism such that $f=\tilde{f} \circ \iota$, as required. 
\end{proof}

Denote the category of the abelian groups by {\bf Ab}.

\begin{prop}
Let $(\A,\N)$ be a Nijenhuis algebra.
\begin{enumerate}
\item Let $(M',\N_{M'})$ be a right $(A,\N)$-module. Then 
there is an additive functor $F_{M'}:  {\bf _{(A,\N)} Mod} \to {\bf Ab}$ defined by
\begin{align*}
F_{M'}(M) = M \ot_{(A,\N)} M' , \quad F_{M'}(g) = \id_{M'} \ot_{(A,\N)} g,
\end{align*}
where $(M, \NM), (M_1, \N_{M_1}) \in  {\bf _{(A,\N)} Mod}$ and $g:(M,\N_{M}) \to (M_1,\N_{M_1})$ is a left $(A,\N)$-module homomorphism.
\mlabel{it:tc1}

\item Let $(M,\NM)$ be a left $(A,\N)$- module. Then
there is an additive functor $G_{M}: {\bf Mod_{(A,\N)}} \to {\bf Ab}$ defined by
\begin{align*}
G_{M}(M') = M \ot_{(A,\N)} M', \quad G_M(g) = g \ot_{(A,\N)} \id_M,
\end{align*}
where $(M',\N_{M'}), (M'_1, \N_{M'_1}) \in {\bf Mod_{(A,\N)}}$ and $g:(M',\N_{M'}) \to (M'_1,\N_{M'_1})$ is a right $(A, \N)$-module homomorphism.
\mlabel{it:tc2}
\end{enumerate}
\mlabel{prop:tc}
\end{prop}

\begin{proof}
(\mref{it:tc1}) For left $(A,\N)$-module homomorphism $g_1:(M_1,\N_{M_1}) \to (M_2,\N_{M_2})$ with $(M_2, \N_{M_2}) \in {\bf _{(A,\N)} Mod}$,
\begin{align*}
F_{M'}(g\circ g_1) = \id_{M'} \ot_{(A,\N)} (g\circ g_1) = (\id_{M'} \ot_{(A,\N)} g)\circ(\id_{M'} \ot_{(A,\N)} g_1) = F_{M'}(g)\circ F_{M'}(g_1). 
\end{align*}
And $F_{M'}(\id_M) = \id_{M} \ot_{(A,\N)} \id_{M'}$, hence $F_{M'}$ is a functor. Moreover, for left $(A,\N)$-module homomorphisms $g,h: (M,\NM) \to (M_1, \N_{M_1})$ and $m \ot_{(A,\N)} m' \in M \ot_{(A,\N)} M'$, we have
\begin{align*}
& F_{M'}(g+h)(m \ot_{(A,\N)} m') = m \ot_{(A,\N)} ((g+h)(m')) \\
&= m \ot_{(A,\N)} (g(m') + h(m')) \\
&= m\ot_{(A,\N)} g(m') + m \ot_{(A,\N)} h(m') \\
&= ( F_{M'}(g) + F_{M'}(h)) (m\ot_{(A,\N)} m'),
\end{align*}
hence $F_{M'}(g+h)=F_{M'}(g)+F_{M'}(h)$. This completes the proof.

(\mref{it:tc2}) The proof is similar to Item~(\mref{it:tc1}).
\end{proof}

Let $(A,\N)$ and $(A',\N_{A'})$ be two Nijenhuis algebras. We call $(M,\NM)$ a {\bf $(A,\N)$-$(A',\N_{A'})$-bimodule} if it is a $A$-$A'$-bimodule, a left $(A,\N)$-module and a right $(A',\N_{A'})$-module.

\begin{prop}\label{TP-m}
Let $(A,\N)$ and $(A',\N')$ be two Nijenhuis algebras.
\begin{enumerate}
\item If $(M,\NM)$ is a $(A,\N)$-$(A',\N_{A'})$-bimodule and $(M', \N_{M'})$ is a left $(A',\N_{A'})$-module,
then $(M \ot_{(A,\N)} M',\N_{\ell})$ is a left $(A,\N)$-module, where
\begin{align*}
a(m \ot_{(A,\N)} m'):= (am) \ot_{(A,\N)} m',\quad
\N_{\ell} (m\ot_{(A,\N)} m'):= \NM(m) \ot_{(A,\N)} m',
\end{align*}
for $a \in A$, $m \in M$ and $m' \in M'$.
\mlabel{it:es1}
\item If $(M,\NM)$ is a right $(A,\N)$-module and $  (M',\N_{M'})$ is a $(A,\N)$-$(A',\N_{A'})$-bimodule,
then $(M \ot_{(A,\N)} M',\N_{r})$ is a right $(A',\N_{A'})$-module, where
\begin{align*}
(m \ot_{(A, \N)} m')a':= m \ot_{(A,\N)} (m'a'), \quad
\N_r (m\ot_{(A,\N)} m') := m \ot_{(A,\N)} \N_{M'}(m'), 
\end{align*}
for $a' \in A'$, $m \in M$ and $m' \in M'$.
\mlabel{it:es2}
\end{enumerate}
\mlabel{prop:es}
\end{prop}

\begin{proof}
(\mref{it:es1}) We know that $M \ot_{(A,\N)} M'$ is a left $A$-module. Moreover, for $a\in A, m\in M, m'\in M'$, we ahve
\begin{align*}
& \N(a) \N_{\ell}(m\ot_{(A,\N)} m')=  \N(a) ( \NM(m) \ot_{(A,\N)} m')\\
=& (\N(a) \NM(m) \ot_{(A,\N)} m'\\
=& (\NM(\N(a)m) + a \NM(m)-\NM(am)) \ot_{(A,\N)} m' \\ 
=& \NM(\N(a) m) \ot_{(A,\N)} m'+ \NM(a \NM(m)) \ot_{(A,\N)} m'- \NM(\NM(am)) \ot_{(A,\N)} m'\\
=& \N_{\ell}(\N(a)(m \ot_{(A,\N)} m'))+ \N_{\ell}(a \N_{\ell}(m \ot_{(A,\N)} m'))-\N_{\ell}(\N_{\ell}(a(m \ot_{(A,\N)} m'))),
\end{align*}
as required.

(\mref{it:es2}) The proof is similar to Item~(\mref{it:es1}).
\end{proof}

\subsection{Flat Nijenhuis modules}
Similar to the classical case, we have that the Nijenhuis tensor product is right exact. To study the exactness of the tensor product, we introduce the concept of flat Nijenhuis modules.

\begin{defn}
Let $(A,\N)$ be a Nijenhuis algebra. A left $(A,\N)$-module $(M',\N_{M'})$ is $\bf flat$ if $- \ot_{(A,\N)} M'$ is an exact functor, that is, for any short exacts sequence of right $(A,\N)$-modules
$$
\xymatrix{
0\ar[r]&(M_1,\N_{M_1})\ar^{i}[r]&(M_2,\N_{M_2})\ar^{j}[r]&(M_3,\N_{M_3})\ar[r]& 0
}
$$
the induced short sequence
$$
\xymatrix{
0\ar[r]& M_1 \ot_{(A,\N)} M' \ar^{i \ot \id_{M'}}[r]& M_2 \ot_{(A,\N)} M' \ar^{ j\ot \id_{M'}}[r]& M_3\ot_{(A,\N)} M'\ar[r]& 0&}
$$
is also exact.
\end{defn}

Thus a left $(A,\N)$-module $(M',\N_{M'})$ is flat if and only if $- \ot_{(A,\N)} M'$ preserves injections.

Let $\{(M_{i},\N_{M_i})~|~i \in I \}$ be a family of left $(A,\N)$-modules. Then $\Big(\bigoplus_{i\in I}M_{i},\N_{\oplus M_i} \Big),$
where $\N_{\oplus M_i}$ is defined by
\begin{align*}
\N_{\oplus M_i}(m_{i})_{I} = (\N_{M_i}(m_{i}))_{I},   
\end{align*}
is also a left $(A,\N)$-module and it is the direct sum of $\{(M_i, \N_{M_i}) \mid i \in I \}$. For each $i\in I$, the natural injection $\iota_{i}: M_{i} \to \bigoplus_{i\in I} M_{i}$ satisfies $ \N_{\oplus M_i} \circ \iota_{i}=\iota_{i}\circ p_{i}$. The natural surjection $\pi_{i}: \bigoplus_{i\in I} M_{i} \to M_{i}$ satisfies $\N_{M_i}\circ \pi_{i}=\pi_{i}\circ \N_{\oplus M_i}$. 

\begin{lemma}\label{dir-sum-injective}
Let $(A, \N)$ be a Nijenhuis algebra and $\varphi_i: (M_{i},\N_{M_i}) \to (M'_{i},\N_{M'_i})$ a family of left $(A,\N)$-module homomorphisms with  $\{(M_{i},\N_{M_i})~|~i \in I \}$, $\{(M'_{i},\N_{M'_i})~|~i \in I \}$ two families of $(A,\N)$-modules. Then the $(A, \N)$-module homomorphism
$$\varphi:= \bigoplus_{i\in I}\varphi_i: \left( \bigoplus_{i\in I} M_{i},\N_{\oplus M_i} \right) \to \left( \bigoplus_{i\in I} M'_{i},\N_{\oplus M'_i} \right), \quad
(m_{i})_{I} \mapsto (\varphi_{i}(m_{i}))_{I},
$$
is injective if and only if each $\varphi_{i}$ is injective.
\end{lemma}
\begin{proof}
This follows from the fact that $\ker \varphi = \oplus_{i\in I} \ker \varphi_i$.
\end{proof}

\begin{lemma}\label{dir-otim-iso}
Let $(A, \N)$ be a Nijenhuis algebra, $ \{(M'_{i},\N_{M'_i})~|~i \in I \}$ a family of left $(A, \N)$-modules, and $(M, \N_{M})$ a right $(A, \N)$-module. Then $M \ot_{(A, \N)}(\oplus_{i\in I}M'_{i})\cong \oplus_{i\in I}(M \ot_{(A, \N) }M'_{i})$.
\end{lemma}
\begin{proof}
Define group homomorphisms
\begin{eqnarray*}
f: M \ot_{(A,\N)}(\oplus_{i\in I}M'_{i}) \to \oplus_{i\in I}(M \ot_{(A,\N)} M'_{i}), \
m \ot_{(A,\N)} (m'_{i})_{I} \mapsto (m \ot_{(A,\N)} m'_{i})_{I},
\end{eqnarray*}
and
\begin{eqnarray*}
g: \oplus_{i\in I}(M \ot_{(A,\N)} M'_{i}) \to M \ot_{(A,\N)}(\oplus_{i\in I}M'_{i}), \
(m_{i}\ot_{(A,\N)} m'_{i})_{I} \mapsto \sum_{i \in I} m_{i} \ot_{(A,\N)} \iota_i(m'_{i}).
\end{eqnarray*}
Then we have that $f\circ g=\id_{\oplus_{i\in I}(M \ot_{(A,\N)} M'_{i})}$ and $g\circ f=\id_{M \ot_{(A,\N)}(\oplus_{i\in I} M'_{i})}$. 
\end{proof}

\begin{prop}\label{dir-sum-flat}
Let $(A,\N)$ be a Nijenhuis algebra and $\{(M'_{i},\N_{M'_i})~|~i \in I \}$  a family of left $(A,\N)$-modules. Then the direct sum $(\bigoplus_{i\in I} M'_{i}, \N_{\oplus M'_i})$ is flat if and only if each $(M'_{i},\N_{M'_i})$ is flat.
\end{prop}
\begin{proof}
Let $f: (M_1, \N_{M_1}) \to (M_2, \N_{M_2})$ be a right $(A,\N)$-module monomorphism with $(M_1,\N_{M_1})$ and $(M_2, \N_{M_2})$  two right $(A,\N)$-modules.

Assume that each left $(A,\N)$-module $(M'_{i},\N_{M'_i})$ is flat. Then each $f \ot_{(A,\N)} \id_{M'_{i}}: M_1 \ot_{(A,\N)} M'_{i} \to M_2 \ot_{(A,\N)} M'_{i}$
is injective. By Lemma~\ref{dir-sum-injective}, the homomorphism
$$\bigoplus_{i\in I}(f \ot_{(A,\N)} \id_{M'_{i}}): \bigoplus_{i\in I}(M_1 \ot_{(A,\N)} M'_{i}) \to \bigoplus_{i\in I}(M_2 \ot_{(A,\N)} M'_{i})$$
is also injective. Hence the $(A, \N)$-module $(\bigoplus_{i\in I} M'_{i},  \N_{\oplus M'_i})$ is flat by Lemma~\ref{dir-otim-iso}.

Conversely, assume that the left $(A,\N)$-module $(\bigoplus_{i\in I} M'_{i},\N_{\oplus M'_i})$ is flat. Then the group homomorphism
$$f \ot_{(A,\N)} \id_{(\bigoplus_{i\in I} M'_{i})}: M_1 \ot_{(A,\N)} \Big(\bigoplus_{i\in I} M'_{i}\Big) \to M_2 \ot_{(A,\N)} \Big(\bigoplus_{i\in I} M'_{i}\Big)$$
is an injective map. By Lemma~\ref{dir-otim-iso}, we have that
$$\bigoplus_{i\in I}(M_1 \ot_{(A,\N)} M'_{i}) \to \bigoplus_{i\in I}( M_2 \ot_{(A,\N)} M'_{i})$$
is also an injective map. By Lemma~\ref{dir-sum-injective}, for each $i\in I$, the map
$$
f \ot_{(A,\N)} \id_{M'_{i}}:  M_1 \ot_{(A,\N)} M'_{i} \to M_2 \ot_{(A,\N)} M'_{i}
$$
is injective. Thus each left $(A,\N)$-module $(M'_{i},\N_{M'_i})$ is flat.
\end{proof}

\begin{theorem}
Let $(A,\N)$ be a Nijenhuis algebra. Then every free left $(A,\N)$-module is flat.
\mlabel{thm:flat}
\end{theorem}

\begin{proof}
Let $(M(X)/I_X,\overline{P}_X)$ be the free left $(A,\N)$-module generated by $X$ as constructed in Theorem~\mref{thm:freem}. For any right $(A,\N)$-module monomorphism $(M_1,\N_{M_1})\to (M_2,\N_{M_2})$, we need to show that the induced map $M_1 \ot_{(A,\N)} (M(X)/I_X) \to M_2 \ot_{(A,\N)} (M(X)/I_X)$ is also an abelian group monomorphism.

To prove this, we show that for any right $(A,\N)$-module $(M,\NM)$, there is an isomorphism
\begin{align*}
M \ot_{(A,\N)} (M(X)/I_X) \cong \oplus_{x \in X} M.
\end{align*}

 First, for singleton set $X=\{ x \}$, we have 
\begin{align*}
M(\{ x\})/I_X= \left( \oplus_{n \geq 1} A^{\ot n} \right)x/I_{\{ x\}}.
\end{align*}
Define maps 
\begin{align*}
f: M \ot_{(A,\N)} \left( \left(\oplus_{n \geq 1} A^{\ot n} x\right)/I_{\{ x\}} \right) \to M, \quad m \ot_{(A, \N)} ((a_1 \ot \cdots \ot a_n) x+ I_{\{ x\}}) \mapsto ma_1\cdots a_n
\end{align*}
and
\begin{align*}
g: M \to M \ot_{(A,\N)} \left( \left(\oplus_{n \geq 1} A^{\ot n} x\right)/I_{\{ x\}} \right), \quad m \mapsto m \ot_{(A,\N)} (1_{A} \ot x+ I_{\{x \} }).
\end{align*}
Then we have that $f\circ g=\id_{M}$ and $g \circ f=\id_{M\ot_{(A,\N)}((\oplus_{n\geq 1} A^{\otimes n})x/I_{\{x\}})}$. Thus $M \ot_{(A, \N)} (M(\{ x\})/I_{\{x\} }) \cong M$.

For any set $X$, by noting $M(X)/I_X \cong \bigoplus_{x\in X} M(\{x\})/I_{\{x\}})$, we have
\begin{align*}
& M\ot_{(A,\N)} (M(X)/I_X) \cong M\ot_{(A,\N)} \Big(\bigoplus_{x\in X} M(\{x\})/I_{\{x\}}\Big)\\ 
 \cong & \bigoplus_{x\in X}(M\ot_{(A,\N)} M(\{x\})/I_{\{x\}}) \quad (\text{by Lemma~\ref{dir-otim-iso}}) \\
\cong & \bigoplus_{x\in X} M.
\end{align*}

Consequently, $M_1\ot_{(A,\N)} (M(X)/I_X)\cong \oplus_{x\in X} M_1$, and $M_2\ot_{(A,\N)}(M(X)/I_X)\cong \oplus_{x\in X} M_2$. Then by Lemma~\ref{dir-sum-injective}, the group homomorphism $\oplus_{x\in X} M_1 \to \oplus_{x\in X} M_2$ is injective. Hence the free left $(A,\N)$-module $(M(X)/I_X, \overline{P}_X)$ is flat.
\end{proof}

\begin{lemma}\label{proj-free}
Let $(A,\N)$ be a Nijenhis algebra. Then every projective left $(A,\N)$-module is a direct summand of a free $(A,\N)$-module.
\end{lemma}

\begin{proof}
Let $(M,\NM)$ be a projective left $(A,\N)$-module. Denote by $(F(M), \overline{P}_M)$ with the map $j: M \to F(M)$  the free left $(A,\N)$-module generated by the underlying set $M$. 

By the freeness of $(F(M), \overline{P}_M)$, there is a left $(A,\N)$-module epimorphism $f:(F(M), \overline{P}_M) \to (M, \NM)$ such that $f \circ j=\id_{M}$.  Then by the projectivity of $(M. \NM)$, for the epimorphism $f:(F(M), \overline{P}_M) \to (M,\NM)$, there is a left $(A,\N)$-module homomorphism $g:(M, \NM) \to (F(M), \overline{P}_M)$ such that $f \circ g=\id_{M}$. Hence
\begin{align*}
F(M)=\IIm g \oplus \Ker f \cong M \oplus \Ker f,
\end{align*}
as required.
\end{proof}

By Proposition~\ref{dir-sum-flat}, Theorem~\ref{thm:flat} and Lemma~\ref{proj-free}, we obtain the following conclusion.

\begin{theorem}
Let $(A,\N)$ be a Nijenhuis algebra. Then every projective left $(A,\N)$-module is flat.
\mlabel{thm:enouf}
\end{theorem}

Theorem~\mref{thm:enouf} shows that there are enough flat Nijenhuis modules, allowing us to define the derived tensor functors.

\noindent {\bf Acknowledgements}: This work was supported by the National Natural Science Foundation of China (Grant No.~12301025).

\noindent {\bf Author Contributions}: All the authors contributed equally to this work.

\noindent {\bf Data Availability}: No datasets were generated or analysed during the current study.

\vspace{0.5cm}

\hspace{-0.4cm}{\bf Declarations}

\noindent {\bf Conflict of interest}: The authors state that there is no conflict of interest.

\noindent {\bf Competing interests}: The authors declare no competing interests.

\end{document}